\documentclass[12pt,a4paper]{article}
\usepackage[latin9]{inputenc}
\usepackage{amsmath}
\usepackage{amssymb}

\usepackage{xcolor}

\makeatletter

\newcommand{\lyxaddress}[1]{
	\par {\raggedright #1
	\vspace{1.4em}
	\noindent\par}
}

\@ifundefined{date}{}{\date{}}
\usepackage[active]{srcltx}

\usepackage{amsfonts}
\usepackage{amsthm}
\usepackage{mathrsfs}

\newtheorem{theorem}{Theorem}
\newtheorem{proposition}[theorem]{Proposition}
\newtheorem{lemma}[theorem]{Lemma}

\newtheorem*{lemma*}{Lemma}

\theoremstyle{remark}
\newtheorem{remark}[theorem]{Remark}
\newtheorem*{remark*}{Remark}
\newtheorem*{remarks*}{Remarks}
\newtheorem*{example*}{Example}

\theoremstyle{definition}

\newtheorem*{definition*}{Definition}

\newtheorem*{notation*}{Notation}

\newcommand{\Spec}{\mathop\mathrm{spec}\nolimits}
\renewcommand{\Re}{\mathop\mathrm{Re}\nolimits}

\newcommand{\sn}{\mathop\mathrm{sn}\nolimits}
\newcommand{\cn}{\mathop\mathrm{cn}\nolimits}
\newcommand{\dn}{\mathop\mathrm{dn}\nolimits}

\newcommand*{\pFqskip}{8mu} % Macro for hypergeometric series
\catcode`,\active
\newcommand*{\pFq}{\begingroup
        \catcode`\,\active
        \def ,{\mskip\pFqskip\relax}%
        \dopFq
}
\catcode`\,12
\def\dopFq#1#2#3#4#5{%
        {}_{#1}F_{#2}\kern-.1em\biggl(\kern-.1em\genfrac..{0pt}{}{#3}{#4}\biggl|\,#5\kern-.1em\biggr)%
        \endgroup
}

\newcommand{\HG}[3]{\pFq{2}{1}{#1}{#2}{#3}}

\makeatother

\begin{document}
\title{New explicitly diagonalizable Hankel matrices related to the Stieltjes-Carlitz
polynomials}
\author{František~\v{S}tampach$^{1}$, Pavel~\v{S}\v{t}ov\'\i\v{c}ek$^{2}$}
\date{{}}
\maketitle

\lyxaddress{$^{1}$Department of Applied Mathematics, Faculty of Information
Technology, Czech Technical University in~Prague, Th\'akurova~9,
160~00 Praha, Czech Republic}

\lyxaddress{$^{2}$Department of Mathematics, Faculty of Nuclear Sciences and
Physical Engineering, Czech Technical University in Prague, Trojanova
13, 12000 Praha, Czech Republic}
\begin{abstract}
\noindent Four new examples of explicitly diagonalizable Hankel matrices
depending on a parameter $k\in(0,1)$ are presented. The Hankel matrices
are regarded as matrix operators on the Hilbert space $\ell^{2}(\mathbb{N}_{0})$
and the solution of the spectral problem is based on an application
of the commutator method. Each of the Hankel matrices commutes with
a Jacobi matrix which is related to a particular family of the Stieltjes-Carlitz
polynomials. More examples of explicitly diagonalizable structured
matrix operators are obtained when taking into account also weighted
Hankel matrices.
\end{abstract}
\vskip\baselineskip\emph{Keywords}: Hankel matrix, Jacobi matrix,
commutator method, the Stieltjes-Carlitz polynomials

\section{Introduction}

To the authors' best knowledge, the generalized Hilbert matrix
\[
H(\theta)_{m,n}:=1/(m+n+\theta),
\]
with $m,n\in\mathbb{N}_{0}$ and $\theta\in\mathbb{R}\setminus(-\mathbb{N}_{0})$
being a parameter, is the only known example of a Hankel matrix which
is explicitly diagonalizable if regarded as a self-adjoint matrix
operator on the Hilbert space $\ell^{2}(\mathbb{N}_{0})$. A solution
of the spectral problem for the Hilbert matrix $H(0)$ was known already
to Magnus \cite{Magnus}. Later on Rosenblum described in \cite{Rosenblum}
an explicit diagonalization of the generalized Hilbert matrix $H(\theta)$.
Let us recall that $H(\theta)$ always represents a bounded operator
on $\ell^{2}(\mathbb{N}_{0})$, its singular continuous spectrum is
always empty, the absolutely continuous spectrum always fills the
interval $[\,0,\pi\,]$, and for $\theta<1/2$ the point spectrum
of $H(\theta)$ is non-empty with the only possible eigenvalues $\pm\pi/\sin(\pi\theta)$
whose multiplicities are finite and depend on $\theta$.

Rosenblum's approach to the solution of the spectral problem is quite
universal and is based on the powerful commutator method. He showed
the matrix operator $H(\theta)$ to be unitarily equivalent to an
integral operator on $L^{2}(\mathbb{R}_{+},\text{d}x)$ and found
a Sturm-Liouville operator on $\mathbb{R}_{+}$ commuting with that
integral operator. Moreover, the Sturm-Liouville operator turned out
to be explicitly diagonalizable with a simple spectrum. The desired
result then followed rather straightforwardly.

The commutator method can be effectively used in various similar situations.
For instance, a systematic application of the method to Hankel integral
operators can be found in \cite{Yafaev}. For our purposes it is substantial
to note that it is possible to avoid an intermediate step in Rosenblum's
solution when $H(\theta)$ is transformed to an integral operator.
As discussed in detail in \cite{kalvoda-stovicek_lma16}, there exists
a Jacobi matrix $J(\theta)$ commuting with $H(\theta)$. Moreover,
$J(\theta)$ has quite nice properties since the associated orthogonal
polynomial sequence is formed by the dual continuous Hahn polynomials
and as such it is included in the Askey classification scheme. The
corresponding normalized measure of orthogonality is unique (the determinate
case) and is known explicitly \cite{KoekoekLeskySwarttouw}. For $J(\theta)$
this means that it is explicitly diagonalizable if regarded as a matrix
operator on $\ell^{2}(\mathbb{N}_{0})$. Moreover, this is a general
feature that the spectrum of a Jacobi matrix operator is simple. Diagonalization
of $H(\theta)$ is then a direct corollary provided one is able to
evaluate the eigenvalues of $H(\theta)$. It turns out that to this
end an additional piece of information is needed, namely a generating
function for the orthogonal polynomial sequence in question written
in an appropriate form.

The family of explicitly diagonalizable structured matrix operators
can be substantially extended if one considers not only Hankel matrices
but also weighted Hankel matrices. This possibility was systematically
explored in \cite{kalvoda-stovicek_lma16} with the restriction that
the commuting Jacobi matrix $J(a,b,c)$ is still related to the dual
continuous Hahn polynomials which depend on three parameters $a$,
$b$, $c$. In a recent paper \cite{StampachStovicek} other orthogonal
polynomial sequences from the Askey scheme are taken into account
and several new examples of explicitly diagonalizable weighted Hankel
matrices are presented.

In the current paper we still stick to this approach. In a recent study~\cite{StampachStovicek_askey-hankel}, it is shown that the generalized Hankel matrix is the only infinite-rank Hankel matrix which, if regarded as an operator on~$\ell^{2}(\mathbb{N}_{0})$, is diagonalizable by application of the commutator method to Jacobi matrices associated with polynomial families from the Askey scheme. Therefore, when attempting
to find new diagonalizable Hankel matrices, we have to go beyond the Askey
scheme. In this arcticle, we focus on the Stieltjes-Carlitz
polynomials whose basic properties were also well studied. A brief
summary is given below in Subsection~\ref{subsec:Stieltjes-Carlitz-polynomials}.
This means in particular that the considered commuting Jacobi matrices
have rather special form depending on five parameters. Of course,
this implies, too, a restriction on the class of Hankel or weighted
Hankel matrices we wish to explore. Notably, four explicitly diagonalizable
Hankel matrices have been discovered within this class, and this is
the main result presented in the paper. This list is then completed
by additional examples of weighted Hankel matrices with the same property.
In contradiction to the Hilbert matrix, all studied matrix operators
belong to the trace class and therefore they have a pure point spectrum.

Let us now make the settings of the present paper more precise. We
seek Hankel matrices or, more generally, weighted Hankel matrices
commuting with the Jacobi matrix
\begin{equation}
J=\begin{pmatrix}\beta_{0} & \alpha_{0}\\
\alpha_{0} & \beta_{1} & \alpha_{1}\\
 & \alpha_{1} & \beta_{2} & \alpha_{2}\\
 &  & \ddots & \ddots & \ddots
\end{pmatrix}\label{eq:jacobi_mat}
\end{equation}
whose entries are of the form
\begin{equation}
\alpha_{n}:=-\sqrt{(n+1)(n+a+1)(n+b+1)(n+c+1)}\,,\ \beta_{n}:=(k+k^{-1})n\,(n+\sigma).\label{eq:alpha-beta-gen}
\end{equation}
Our choice of the parameters guarantees that $J$ is a non-decomposable
Hermitian matrix, namely $k\in(0,1)$, $\sigma\in\mathbb{R}$, and
$a,b,c>-1$. Later on, however, these parameters will be further specialized
in order to obtain Jacobi operators with an explicitly solvable spectral
problem.

Let us also note that in all cases studied in the sequel the Jacobi
matrix (\ref{eq:jacobi_mat}), (\ref{eq:alpha-beta-gen}) represents
a unique self-adjoint operator on $\ell^{2}(\mathbb{N}_{0})$. This
is why we can afford to be less scrupulous in the notation when we
are using the same symbol for a Jacobi matrix and the corresponding
operator.

The paper is organized as follows. Section~\ref{sec:Preliminaries}
summarizes some preliminary information which is then needed in the
remainder of the paper. An important role in the entire paper is played
by elliptic functions and integrals and this is the subject of Subsection~\ref{subsec:jacobi_elliptic}.
Subsection~\ref{subsec:Stieltjes-Carlitz-polynomials} is devoted
to the Stieltjes-Carlitz polynomials. Section~\ref{sec:Auxiliary-results}
contains some technical auxiliary results which are then used in the
proofs of the presented theorems. In Section~\ref{sec:A-three-term-recurrence},
a three-term recurrence equation is studied with coefficients depending
linearly on the index. The purpose of this study is the fact that
the commutation equation between a Hankel and a Jacobi matrix in our
case finally leads to such a three-term recurrence. This type of equation
is rather general, however, and, as we suppose, it may be encountered
also in other problems. The main goal of Section~\ref{sec:General-commuting-Hankel}
is to determine which Jacobi matrices of the form (\ref{eq:jacobi_mat}),
(\ref{eq:alpha-beta-gen}) admit a nontrivial commuting Hankel matrix.
Section~\ref{sec:The-main-theorem} contains the main result of the
paper, i.e. some examples of explicitly diagonalizable Hankel matrices.
In addition, the list of explicitly diagonalizable structured matrix
operators is extended in Sections~\ref{sec:Families-1-2} and \ref{sec:Some-more-weighted}
by considering also weighted Hankel matrices.

\section{Preliminaries\label{sec:Preliminaries}}

\subsection{Jacobian elliptic functions \label{subsec:jacobi_elliptic}}

We start from recalling the definition of the complete elliptic integrals
of the first kind,
\[
K=K(k):=\frac{\pi}{2}\,\HG{\frac{1}{2},\frac{1}{2}}{1}{k^{2}}=\int_{0}^{\pi/2}\frac{\text{d}\theta}{\sqrt{1-k^{2}\sin^{2}\theta}},\ \,K'(k):=K\big(\sqrt{1-k^{2}}\big),
\]
and the elliptic nome
\[
q=q(k):=\exp\!\big(-\pi K'(k)/K(k)\big),
\]
where $k\in(0,1)$; see, for example,~\cite[Chp.~19]{dlmf}. Then
$q\in(0,1)$.

We shall also need the familiar integral representation of the Gauss
hypergeometric function~\cite[Eq.~15.6.1]{dlmf}
\begin{equation}
\HG{a,b}{c}{z}=\frac{\Gamma(c)}{\Gamma(b)\Gamma(c-b)}\int_{0}^{1}\frac{t^{b-1}(1-t)^{c-b-1}}{(1-tz)^{a}}\,\text{d}t,\label{eq:gauss_hyp_int_repre}
\end{equation}
where $\Re c>\Re b>0$.

A~remark to the notation. In displayed formulas we shall denote the
hypergeometric functions as in equation (\ref{eq:gauss_hyp_int_repre}).
In in-line formulas, however, we prefer the equivalent expression
$\,_{2}F_{1}(a,b;c;z)$.

Further we recall several selected properties of the Jacobian elliptic
functions $\sn(z)=\sn(z,k)$, $\cn(z)=\cn(z,k)$, and $\dn(z)=\dn(z,k)$
that will be needed in the sequel. The reader is referred, for example,
to~\cite{lawden_89} for the theory of the elliptic functions and
to~\cite[Chp.~22]{dlmf} for an easily accessible review of their
fundamental properties.

First, the squares of the Jacobian elliptic functions are mutually
related as follows~\cite[Eq.~22.6.1]{dlmf}
\begin{equation}
\sn^{2}(z)+\cn^{2}(z)=k^{2}\sn^{2}(z)+\dn^{2}(z)=1.\label{eq:sn_cn_dn_square_id}
\end{equation}
Second, we will need the formulas for the first derivatives~\cite[Table~22.13.1]{dlmf}
\begin{equation}
\frac{\text{d}\sn(z)}{\text{d}z}=\cn(z)\dn(z),\ \frac{\text{d}\cn(z)}{\text{d}z}=-\sn(z)\dn(z),\ \frac{\text{d}\dn(z)}{\text{d}z}=-k^{2}\sn(z)\cn(z).\label{eq:deriv_sn_cn_dn}
\end{equation}
Third, we have the special values~\cite[Table~22.5.1]{dlmf}
\begin{equation}
\sn(0)=0,\ \,\cn(0)=1,\ \,\dn(0)=1.\label{eq:sn_cn_dn_0}
\end{equation}
and
\begin{equation}
\sn(K)=1,\ \,\cn(K)=0,\ \,\dn(K)=\sqrt{1-k^{2}}.\label{eq:sn_cn_dn_K}
\end{equation}

Finally, recall the Fourier series~\cite[Eqs.~22.11.1-3]{dlmf}
\begin{equation}
\sn\!\left(\frac{2Kv}{\pi}\right)=\frac{2\pi}{kK}\sum_{n=0}^{\infty}\frac{q^{n+1/2}}{1-q^{2n+1}}\,\sin\big((2n+1)v\big),\label{eq:four_sn}
\end{equation}
\begin{equation}
\cn\!\left(\frac{2Kv}{\pi}\right)=\frac{2\pi}{kK}\sum_{n=0}^{\infty}\frac{q^{n+1/2}}{1+q^{2n+1}}\,\cos\big((2n+1)v\big),\label{eq:four_cn}
\end{equation}
\begin{equation}
\dn\!\left(\frac{Kv}{\pi}\right)=\frac{\pi}{2K}+\frac{2\pi}{K}\sum_{n=1}^{\infty}\frac{q^{n}}{1+q^{2n}}\,\cos(nv),\label{eq:four_dn}
\end{equation}
and also~\cite{kiper_mc84}
\begin{equation}
\sn^{2}\!\left(\frac{Kv}{\pi}\right)=\frac{K-E(k)}{k^{2}K}-\frac{2\pi^{2}}{k^{2}K^{2}}\sum_{n=1}^{\infty}\frac{nq^{n}}{1-q^{2n}}\,\cos(nv),\label{eq:four_sn_squared}
\end{equation}
\begin{equation}
\sn^{3}\!\left(\frac{2Kv}{\pi}\right)=\frac{\pi}{k^{3}K}\sum_{n=0}^{\infty}\frac{q^{n+1/2}}{1-q^{2n+1}}\!\left(1+k^{2}-\frac{(2n+1)^{2}\pi^{2}}{4K^{2}}\right)\!\sin\big((2n+1)v\big),\label{eq:four_sn_cubed}
\end{equation}
where $E$ is the complete elliptic integral of the second kind, see~\cite[Chp.~19]{dlmf}.
The above Fourier expansions hold true for all $v\in\mathbb{R}$ and
$k\in(0,1)$.

\subsection{The Stieltjes-Carlitz polynomials\label{subsec:Stieltjes-Carlitz-polynomials}}

In~\cite{carlitz_dm60}, Carlitz investigated four families of orthogonal
polynomials obtained from certain formulas for continued fractions
of Laplace transform of the Jacobian elliptic functions studied earlier
by Stieltjes and Rogers. Two of the families are symmetric orthogonal
polynomials. Consequently, each of these two families gives rise to
other two families of orthogonal polynomials, see~\cite[Chp.~I, Sec.~9 and Chp.~VI, Sec.~9]{chihara_78}.
In total, there are six families of orthogonal polynomials intimately
related to the Jacobian elliptic functions. For the sake of definiteness,
we call them Family~\#1~-~6 because it seems that there are no
commonly used names for these families in the literature.

Below we list their basic properties that will be needed further.
Namely, we recall the three-term recurrences, orthogonality relations,
and generating functions. These polynomials are defined in their monic
form, i.e., they fulfill a three-term recurrence of the form
\[
P_{n+1}(x)=(x-\beta_{n})P_{n}(x)-\alpha_{n-1}^{\,2}P_{n-1}(x),\,\ n\in\mathbb{N}_{0},
\]
($\alpha_{-1}$ is arbitrary) with the standard initial conditions
$P_{-1}(x)=0$ and $P_{0}(x)=1$. Moreover, all families depend on
a parameter $k$. It is always assumed that $k\in(0,1)$ in which
case every family has a unique measure of orthogonality. In other
words, the respective Hamburger moment problems are all determinate,
see ~\cite[Chp.~VI, Sec.~9]{chihara_78} or~\cite[Sec.~21.9]{ismail09}
and references therein. Consequently, the Jacobi matrices corresponding
to the families of orthogonal polynomials listed below give rise to
unique self-adjoint Jacobi operators; see~\cite{akhiezer90} for
the general theory.

\textbf{Family~\#1}. The three-term recurrence:
\begin{equation}
f_{n+1}(x)=\big(x+(k^{2}+1)(2n+1)^{2}\big)f_{n}(x)-k^{2}(2n-1)(2n)^{2}(2n+1)f_{n-1}(x),\ \,n\geq0.\label{eq:recur_carlitz1}
\end{equation}
The orthogonality relation:
\begin{equation}
\int_{0}^{\infty}f_{n}(-x)f_{m}(-x)\,\text{d}\mu(x)=k^{2n}(2n)!(2n+1)!\delta_{m,n},\ \,m,n\geq0.\label{eq:og_rel_carlitz1}
\end{equation}
where
\begin{equation}
\mu=\frac{\pi^{2}}{K^{2}k}\sum_{m=0}^{\infty}\frac{(2m+1)q^{m+1/2}}{1-q^{2m+1}}\,\delta_{\lambda_{m}}\quad\mbox{ and }\quad\lambda_{m}=\frac{\pi^{2}(2m+1)^{2}}{4K^{2}}.\label{eq:mu_carlitz1}
\end{equation}
Here and below, $\delta_{x}$ denotes the unit-mass Dirac delta measure
supported on the one-point set $\{x\}$.\\
 The generating function:
\begin{equation}
\sum_{n=0}^{\infty}\frac{f_{n}(x)}{(2n+1)!}\sn^{2n+1}(u)=\frac{\sinh(\sqrt{x}u)}{\sqrt{x}}.\label{eq:gener_func_carlitz1}
\end{equation}

\textbf{Family~\#2}. The three-term recurrence:
\begin{equation}
g_{n+1}(x)=\big(x+(k^{2}+1)(2n+2)^{2}\big)g_{n}(x)-k^{2}(2n)(2n+1)^{2}(2n+2)g_{n-1}(x),\ \,n\geq0.\label{eq:recur_carlitz2}
\end{equation}
The orthogonality relation:
\begin{equation}
\int_{0}^{\infty}g_{n}(-x)g_{m}(-x)\,\text{d}\mu(x)=\frac{k^{2n}(2n+1)!(2n+2)!}{2}\,\delta_{m,n},\ \,m,n\geq0,\label{eq:og_rel_carlitz2}
\end{equation}
where
\begin{equation}
\mu=\frac{\pi^{4}}{K^{4}k^{2}}\sum_{m=1}^{\infty}\frac{m^{3}q^{m}}{1-q^{2m}}\,\delta_{\lambda_{m}}\quad\mbox{ and }\quad\lambda_{m}=\frac{\pi^{2}m^{2}}{K^{2}}.\label{eq:mu_carlitz2}
\end{equation}
The generating function:
\begin{equation}
\sum_{n=0}^{\infty}\frac{g_{n}(x)}{(2n+1)!}\sn^{2n+1}(u)=\frac{\sinh(\sqrt{x}u)}{\sqrt{x}\cn(u)\dn(u)}.\label{eq:gener_func_carlitz2}
\end{equation}

\textbf{Families~\#3 and~\#4}. The three-term recurrences:
\begin{equation}
p_{n+1}(x)=\big(x-k^{2}(2n)^{2}-(2n+1)^{2}\big)p_{n}(x)-k^{2}(2n)^{2}(2n-1)^{2}p_{n-1}(x),\ \,n\geq0,\label{eq:recur_carlitz3}
\end{equation}
and
\begin{equation}
q_{n+1}(x)=\big(x-(2n+1)^{2}-k^{2}(2n+2)^{2}\big)q_{n}(x)-k^{2}(2n+1)^{2}(2n)^{2}q_{n-1}(x),\ \,n\geq0.\label{eq:recur_carlitz4}
\end{equation}
The orthogonality relations:
\begin{equation}
\int_{0}^{\infty}p_{n}(x)p_{m}(x)\,\text{d}\mu(x)=k^{2n}((2n)!)^{2}\delta_{m,n},\ \,m,n\geq0,\label{eq:og_rel_carlitz3}
\end{equation}
and
\begin{equation}
\int_{0}^{\infty}q_{n}(x)q_{m}(x)x\,\text{d}\mu(x)=k^{2n}((2n+1)!)^{2}\delta_{m,n},\ \,m,n\geq0,\label{eq:og_rel_carlitz4}
\end{equation}
where
\begin{equation}
\mu=\frac{2\pi}{Kk}\sum_{m=0}^{\infty}\frac{q^{m+1/2}}{1+q^{2m+1}}\,\delta_{\lambda_{m}}\quad\mbox{ and }\quad\lambda_{m}=\frac{\pi^{2}(2m+1)^{2}}{4K^{2}}.\label{eq:mu_carlitz34}
\end{equation}
The generating functions:
\begin{equation}
\sum_{n=0}^{\infty}\frac{(-1)^{n}p_{n}(x)}{(2n)!}\sn^{2n}(u)=\frac{\cos(\sqrt{x}u)}{\cn(u)}\label{eq:gener_func_carlitz3}
\end{equation}
and
\begin{equation}
\sum_{n=0}^{\infty}\frac{(-1)^{n}q_{n}(x)}{(2n+1)!}\sn^{2n+1}(u)=\frac{\sin(\sqrt{x}u)}{\sqrt{x}\dn(u)}.\label{eq:gener_func_carlitz4}
\end{equation}

\textbf{Families~\#5 and~\#6}. The three-term recurrences:
\begin{equation}
r_{n+1}(x)=\big(x-(2n)^{2}-k^{2}(2n+1)^{2}\big)r_{n}(x)-k^{2}(2n)^{2}(2n-1)^{2}r_{n-1}(x),\ \,n\geq0,\label{eq:recur_carlitz5}
\end{equation}
and
\begin{equation}
s_{n+1}(x)=\big(x-k^{2}(2n+1)^{2}-(2n+2)^{2}\big)s_{n}(x)-k^{2}(2n+1)^{2}(2n)^{2}s_{n-1}(x),\ \,n\geq0.\label{eq:recur_carlitz6}
\end{equation}
The orthogonality relations:
\begin{equation}
\int_{0}^{\infty}r_{n}(x)r_{m}(x)\,\text{d}\mu(x)=k^{2n}((2n)!)^{2}\delta_{m,n},\ \,m,n\geq0,\label{eq:og_rel_carlitz5}
\end{equation}
and
\begin{equation}
\int_{0}^{\infty}s_{n}(x)s_{m}(x)x\,\text{d}\mu(x)=k^{2n+2}((2n+1)!)^{2}\delta_{m,n},\ \,m,n\geq0,\label{eq:og_rel_carlitz6}
\end{equation}
where
\begin{equation}
\mu=\frac{2\pi}{K}\sum_{m=0}^{\infty}\frac{q^{m}}{1+q^{2m}}\,\delta_{\lambda_{m}}\quad\mbox{ and }\quad\lambda_{m}=\frac{\pi^{2}m^{2}}{K^{2}}.\label{eq:mu_carlitz56}
\end{equation}
The generating functions:
\begin{equation}
\sum_{n=0}^{\infty}\frac{(-1)^{n}r_{n}(x)}{(2n)!}\,\sn^{2n}(u)=\frac{\cos(\sqrt{x}u)}{\dn(u)}\label{eq:gener_func_carlitz5}
\end{equation}
and
\begin{equation}
\sum_{n=0}^{\infty}\frac{(-1)^{n}s_{n}(x)}{(2n+1)!}\sn^{2n+1}(u)=\frac{\sin(\sqrt{x}u)}{\sqrt{x}\cn(u)}.\label{eq:gener_func_carlitz6}
\end{equation}

\section{Several auxiliary results\label{sec:Auxiliary-results}}

The asymptotic expansion to the leading order of the Stieltjes-Carlitz
polynomials can be derived in a comparatively straightforward way
and this fact was already exploited by some authors \cite{ismail-etal_jat01}.
For the sake of completeness and since the asymptotic formulas will
be of some importance in the sequel, the result is presented here,
too.

\begin{proposition}\label{thm:asympt-SC-poly} The leading terms
in the asymptotic expansion of the Stieltjes-Carlitz polynomials $p_{n}(x)$,
$q_{n}(x)$, $r_{n}(x)$ and $s_{n}(x)$, defined in (\ref{eq:recur_carlitz3}),
(\ref{eq:recur_carlitz4}), (\ref{eq:recur_carlitz5}) and (\ref{eq:recur_carlitz6}),
respectively, are as follows
\begin{eqnarray*}
\frac{p_{n}(x)}{(2n)!} & = & \frac{(-1)^{n}}{\sqrt{\pi n}}\,\cos\!\big(\sqrt{x}K\big)+o\!\left(\frac{1}{n}\right)\!,\\
\frac{q_{n}(x)}{(2n+1)!} & = & \frac{(-1)^{n}}{2(1-k^{2})\sqrt{\pi}\,n^{3/2}}\,\cos\!\big(\sqrt{x}K\big)+o\!\left(\frac{1}{n^{2}}\right)\!,\\
\frac{r_{n}(x)}{(2n)!} & = & \frac{(-1)^{n+1}}{2(1-k^{2})\sqrt{\pi}\,n^{3/2}}\,\sqrt{x}\sin\!\big(\sqrt{x}K\big)+o\!\left(\frac{1}{n^{2}}\right)\!,\\
\frac{s_{n}(x)}{(2n+1)!} & = & \frac{(-1)^{n}}{\sqrt{\pi n}}\,\frac{\sin\!\big(\sqrt{x}K\big)}{\sqrt{x}}+o\!\left(\frac{1}{n}\right)\!,
\end{eqnarray*}
 as $n\to\infty$. Here $x$ is an arbitrary fixed complex number.
\end{proposition}

\begin{proof} The generating function~\eqref{eq:gener_func_carlitz3}
can be rewritten as
\[
\sum_{n=0}^{\infty}\frac{(-1)^{n}p_{n}(x)}{(2n)!}\,\xi^{n}=\frac{\cos\!\left(\sqrt{x}\sn^{-1}(\sqrt{\xi})\right)}{\sqrt{1-\xi}}=:g(\xi).
\]
The singularity of $g(\xi)$ located most closely to the origin occurs
at $\xi=1$. One can apply the Darboux method with the comparison
function
\[
g_{c}(\xi):=\frac{\cos(\sqrt{x}K)}{\sqrt{1-\xi}}\,,
\]
see~\cite[Sec.~8.9]{olver_97} and especially the refinements in
$\S\,8.9.3$. In order to get the asymptotic formula for $p_{n}(x)$
it suffices to observe that the restriction of $g(\xi)-g_{c}(\xi)$
to the unite circle $\xi=e^{i\theta}$, $\theta\in[\,0,2\pi\,]$,
is continuous. Moreover, the restriction of the derivative $g'(\xi)-g_{c}'(\xi)$
to the unite circle is continuous, too, except singularities at $\theta=0$
and $\theta=2\pi$ which are integrable, however. More precisely,
the singularity at $\theta=0$ is of order $\theta^{-1/2}$, and similarly
for $\theta=2\pi$.

As is the main idea of the Darboux method, these feature make it possible
to effectively compare the coefficients in the power series expansions
of $g(\xi)$ and $g_{c}(\xi)$. We skip further details of this standard
approach.

Very analogously one can proceed in the case of polynomials $q_{n}(x)$,
$r_{n}(x)$ and $s_{n}(x)$. Omitting additional details we confine
ourselves to pointing out equations to which the Darboux method can
be applied. By differentiating (\ref{eq:gener_func_carlitz4}) and
using the rules (\ref{eq:sn_cn_dn_square_id}) and (\ref{eq:deriv_sn_cn_dn})
we obtain
\[
\sum_{n=0}^{\infty}\frac{(-1)^{n}q_{n}(x)}{(2n)!}\,\xi^{n}=\frac{\cos\!\left(\sqrt{x}\sn^{-1}(\sqrt{\xi})\right)}{\sqrt{1-\xi}\,(1-k^{2}\xi)}+\frac{k^{2}\sin\!\left(\sqrt{x}\sn^{-1}(\sqrt{\xi})\right)\!\sqrt{\xi}}{\sqrt{x}\,(1-k^{2}\xi)^{3/2}}\,.
\]
Similarly one can treat equation (\ref{eq:gener_func_carlitz5}) to
get
\[
\sum_{n=1}^{\infty}\frac{(-1)^{n}r_{n}(x)}{(2n-1)!}\,\xi^{n}=-\frac{\sqrt{x\xi}\sin\!\left(\sqrt{x}\sn^{-1}(\sqrt{\xi})\right)}{\sqrt{1-\xi}\,(1-k^{2}\xi)}+\frac{k^{2}\cos\!\left(\sqrt{x}\sn^{-1}(\sqrt{\xi})\right)\!\xi}{(1-k^{2}\xi)^{3/2}}\,.
\]
Finally, equation (\ref{eq:gener_func_carlitz6}) can be quite straightforwardly
rewritten as
\[
\sum_{n=0}^{\infty}\frac{(-1)^{n}s_{n}(x)}{(2n+1)!}\,\xi^{n}=\frac{\sin\!\left(\sqrt{x}\sn^{-1}(\sqrt{\xi})\right)}{\sqrt{x\xi}\sqrt{1-\xi}}\,.
\]
The desired asymptotic formulas follow. \end{proof}

\begin{lemma}\label{thm:asympt-int-sn-cn-dn} We have
\begin{eqnarray*}
\int_{0}^{K}\sn^{2n}(u)\cn^{2}(u)\,\text{d}u & = & \sqrt{\frac{\pi}{1-k^{2}}}\,\frac{1}{4n^{3/2}}\,\big(1+o(1)\big),\\
\noalign{\smallskip}\int_{0}^{K}\sn^{2n}(u)\dn^{2}(u)\,\text{d}u & = & \frac{\sqrt{(1-k^{2})\pi}}{2n^{1/2}}\,\big(1+o(1)\big),\ \,\text{as}\ n\to\infty.
\end{eqnarray*}
 \end{lemma}

\begin{proof} The former integral, if written in the form
\[
\int_{0}^{K}e^{-np(u)}q(u)\,\text{d}u,
\]
with $p(u):=-2\ln(\sn(u))$ and $q(u):=\cn^{2}(u)$, admits a direct
application of the Laplace method, see see for instance \cite[Sec.~3.7]{olver_97}.
Note that $p(u)$ is strictly decreasing for $u\in(0,K]$ and
\begin{eqnarray*}
p(u) & = & (1-k^{2})(u-K)^{2}+O\big((u-K)^{4}\big),\\
q(u) & = & (1-k^{2})(u-K)^{2}+O\big((u-K)^{4}\big),\ \,\text{as}\ u\to K.
\end{eqnarray*}
In the case of the latter integral we keep the function $p(u)$ but
now we let
\[
q(u):=\dn^{2}(u)=1-k^{2}+O\big((u-K)^{2}\big),\ \,\text{as}\ u\to K.
\]
Again, the Laplace method gives the result. \end{proof}

\begin{proposition}\label{thm:sum-SCpoly-3456} For $x\in\mathbb{C}$
and the Stieltjes-Carlitz polynomials $p_{n}(x)$, $q_{n}(x)$, $r_{n}(x)$
and $s_{n}(x)$, defined in (\ref{eq:recur_carlitz3}), (\ref{eq:recur_carlitz4}),
(\ref{eq:recur_carlitz5}) and (\ref{eq:recur_carlitz6}), respectively,
it holds true that
\begin{eqnarray}
\sum_{n=0}^{\infty}\frac{(-1)^{n}E_{n}(k)}{(2n)!}\,p_{n}(x) & = & \int_{0}^{K}\cos(\sqrt{x}u)\cn(u)\,\text{d}u,\label{eq:h3_sum_form}\\
\sum_{n=0}^{\infty}\frac{(-1)^{n}F_{n+1}(k)}{(2n+1)!}\,q_{n}(x) & = & \int_{0}^{K}\cos(\sqrt{x}u)\cn(u)\,\text{d}u,\label{eq:h4_sum_form}\\
\sum_{n=0}^{\infty}\frac{(-1)^{n}F_{n}(k)}{(2n)!}\,r_{n}(x) & = & \int_{0}^{K}\cos(\sqrt{x}u)\dn(u)\,\text{d}u,\label{eq:h5_sum_form}\\
\sum_{n=0}^{\infty}\frac{(-1)^{n}E_{n+1}(k)}{(2n+1)!}\,s_{n}(x) & = & -\frac{\sqrt{1-k^{2}}\sin(\sqrt{x}K)}{k^{2}\sqrt{x}}+\frac{1}{k^{2}}\int_{0}^{K}\cos(\sqrt{x}u)\dn(u)\,\text{d}u,\nonumber \\
\label{eq:h6_sum_form}
\end{eqnarray}
where
\begin{eqnarray}
\hskip-1.7emE_{n}(k) & := & \int_{0}^{1}t^{2n}\sqrt{\frac{1-t^{2}}{1-k^{2}t^{2}}}\,\text{d}t\,=\,\frac{\pi(2n)!}{2^{2n+2}n!(n+1)!}\,\HG{n+\frac{1}{2},\frac{1}{2}}{n+2}{k^{2}},\label{eq:En}\\
\noalign{\smallskip}\hskip-1.7emF_{n}(k) & := & \int_{0}^{1}t^{2n}\sqrt{\frac{1-k^{2}t^{2}}{1-t^{2}}}\,\text{d}t\,=\,\frac{\pi(2n)!}{2^{2n+1}(n!)^{2}}\,\HG{n+\frac{1}{2},-\frac{1}{2}}{n+1}{k^{2}},\label{eq:Fn}
\end{eqnarray}
$n\geq0$ (the latter equations in (\ref{eq:En}), (\ref{eq:Fn})
follow from (\ref{eq:gauss_hyp_int_repre})). \end{proposition}

\begin{proof} Substitution $t=\sn(u)$ in the integral (\ref{eq:En})
brings the LHS of (\ref{eq:h3_sum_form}) to the form 
\[
\sum_{n=0}^{\infty}\frac{(-1)^{n}p_{n}(x)}{(2n)!}\,\int_{0}^{K}\sn^{2n}(u)\cn^{2}(u)\,\text{d}u.
\]
Then after interchanging the integral and the sum and using the generating
function (\ref{eq:gener_func_carlitz3}) one arrives at the RHS of
(\ref{eq:h3_sum_form}). The interchanging of summation and integration
is justified by the Fubini theorem and the respective asymptotic formulas
in Proposition~\ref{thm:asympt-SC-poly} and Lemma~\ref{thm:asympt-int-sn-cn-dn}.

Analogously, substitution $t=\sn(u)$ in the integral (\ref{eq:Fn})
brings the LHS of (\ref{eq:h4_sum_form}) to the form
\[
\sum_{n=0}^{\infty}\frac{(-1)^{n}q_{n}(x)}{(2n+1)!}\,\int_{0}^{K}\sn^{2n+2}(u)\dn^{2}(u)\,\text{d}u.
\]
By interchanging the integral and the sum and using the generating
function (\ref{eq:gener_func_carlitz4}) one obtains the expression
\[
\int_{0}^{K}\frac{\sin(\sqrt{x}u)}{\sqrt{x}}\sn(u)\dn(u)\,\text{d}u.
\]
Integrating by parts, using (\ref{eq:deriv_sn_cn_dn}) and the special
values~\eqref{eq:sn_cn_dn_0}, \eqref{eq:sn_cn_dn_K} leads to the
RHS of (\ref{eq:h4_sum_form}). The interchanging of summation and
integration is again possible owing to the respective asymptotic formulas
in Proposition~\ref{thm:asympt-SC-poly} and Lemma~\ref{thm:asympt-int-sn-cn-dn}.

With the aid of the same substitution as above the LHS of (\ref{eq:h5_sum_form})
is transformed to
\[
\sum_{n=0}^{\infty}\frac{(-1)^{n}r_{n}(x)}{(2n)!}\,\int_{0}^{K}\sn(u)^{2n}\dn^{2}(u)\,\text{d}u.
\]
Relying on Proposition~\ref{thm:asympt-SC-poly} and Lemma~\ref{thm:asympt-int-sn-cn-dn}
one can interchange the integral and the sum, and using the generating
function (\ref{eq:gener_func_carlitz5}) one arrives at the RHS of
(\ref{eq:h5_sum_form}).

Very analogously as in the foregoing equations the LHS of (\ref{eq:h6_sum_form})
is shown to be equal to
\[
\sum_{n=0}^{\infty}\frac{(-1)^{n}s_{n}(x)}{(2n+1)!}\,\int_{0}^{K}\sn^{2n+2}(u)\cn^{2}(u)\,\text{d}u.
\]
Interchanging the integral and the sum is again justifiable, and using
the generating function~\eqref{eq:gener_func_carlitz6} one obtains
the expression
\[
\int_{0}^{K}\frac{\sin(\sqrt{x}u)}{\sqrt{x}}\sn(u)\cn(u)\,\text{d}u.
\]
Now we can integrate by parts while taking into account (\ref{eq:deriv_sn_cn_dn}),
(\ref{eq:sn_cn_dn_0}) and (\ref{eq:sn_cn_dn_K}), and we get the
desired identity. \end{proof}

\section{A three-term recurrence equation with coefficients depending linearly
on the index\label{sec:A-three-term-recurrence}}

The problem of finding Hankel matrices commuting with a given Jacobi
matrix of the form (\ref{eq:jacobi_mat}) finally leads to a three-term
recurrence equation. This section is devoted to a basic study of such
an equation in its own right. The coefficients in the equation are
of a particular form which is dictated by the intended application.

We will discuss the three-term recurrence 
\begin{equation}
(k+k^{-1})(n+\sigma)h_{n}-(n+\xi)h_{n-1}-(n+\eta)h_{n+1}=0,\label{eq:three-term-eq}
\end{equation}
$n\geq1$, where $k\in(0,1)$ and $\sigma,\xi,\eta\in\mathbb{C}$
are parameters. We claim that, up to a constant multiplier, equation
(\ref{eq:three-term-eq}) has exactly one square summable solution
$(h_{n})_{n\geq0}$. One can show that this is true even for a somewhat
more general type of equation.

Denote by $\mathbf{e}_{n}$, $n=0,1,2,\ldots$, the semi-infinite
column vectors with all zero entries except a unit on the $n$th position
(counting from $0$ upwards). Furthermore, $\mathbf{k}$ is another
semi-infinite column vector,
\[
\mathbf{k}:=(1,k,k^{2},\ldots)^{T}.
\]
Further we introduce two semi-infinite matrices, $L$ and $G$, defined
as follows,
\begin{equation}
L_{m,n}:=(k+k^{-1})\delta_{m,n}-\delta_{m,n+1}-\delta_{m+1,n},\ G_{m,n}:=\frac{k^{|m-n|+1}}{1-k^{2}}\,,\ \,m,n\in\mathbb{N}_{0}.\label{eq:L-G}
\end{equation}
The matrices satisfy the equations
\begin{equation}
LG=I+\frac{k^{2}}{1-k^{2}}\,\mathbf{e}_{0}\hskip0.03em\mathbf{k}^{T},\ GL=I+\frac{k^{2}}{1-k^{2}}\,\mathbf{k}\hskip0.06em\mathbf{\mathbf{e}}_{0}^{\,T},\label{eq:LG_GL}
\end{equation}
where $I$ is the unit matrix. The matrix product of $L$ and $G$
makes good sense since $L$ is a band matrix.

\begin{proposition}\label{thm:hn_l2} Let us consider the three-term
recurrence equation
\begin{equation}
(k+k^{-1})(1+s_{n})h_{n}-(1+x_{n})h_{n-1}-(1+y_{n})h_{n+1}=0,\ \,n\geq1,\label{eq:threetermeq-gen}
\end{equation}
where $k\in(0,1)$ and $(s_{n})_{n\geq1}$, $(x_{n})_{n\geq1}$ and
$(y_{n})_{n\geq1}$ are given complex sequences. Assume that $1+x_{n}\neq0$
for all $n\geq1$. If
\[
\lim_{n\to\infty}s_{n}=\lim_{n\to\infty}x_{n}=\lim_{n\to\infty}y_{n}=0
\]
then, up to a constant multiplier, there exists exactly one solution
to (\ref{eq:threetermeq-gen}) which is square summable. \end{proposition}

\begin{proof} It suffices to consider equation (\ref{eq:threetermeq-gen})
on a neighborhood of $\infty$ determined by a lower bound $N\in\mathbb{N}$.
We can choose $N$ sufficiently large and this will be specified more
precisely later on. Let us denote
\[
\rho_{n}:=\sup_{k\geq n}\,\max\{|s_{k}|,|x_{k}|,|y_{k}|\}.
\]
By the assumption, $\lim_{n\to\infty}\rho_{n}=0$.

(I)~Suppose $(h_{n})_{n\geq0}$ is a square summable solution to
(\ref{eq:threetermeq-gen}). Let
\[
\tilde{h}_{n}:=h_{N+n},\ n\in\mathbb{N}_{0}.
\]
Then $(\tilde{h}_{n})_{n\geq0}$ solves the equation
\begin{equation}
(k+k^{-1})\tilde{h}_{n}-\tilde{h}_{n-1}-\tilde{h}_{n+1}=\big(-(k+k^{-1})s_{N+n}\tilde{h}_{n}+x_{N+n}\tilde{h}_{n-1}+y_{N+n}\tilde{h}_{n+1}\big),\ n\geq1.\label{eq:threetermeq-rewritten}
\end{equation}
The equation can be rewritten in terms of matrices and vectors. Let
us introduce a semi-infinite matrix $R$,
\[
R_{m,n}:=\big(-(k+k^{-1})s_{N+m}\delta_{m,n}-x_{N+m}\delta_{m,n+1}-y_{N+m}\delta_{m+1,n}\big),\ \,m,n\in\mathbb{N}_{0}.
\]
Then (\ref{eq:threetermeq-rewritten}) means that $(L\tilde{h})_{m}=(R\tilde{h})_{m}$
for $m\geq1$. Hence, in view of (\ref{eq:LG_GL}), we have for all
$n\geq0$,
\[
\tilde{h}_{n}=(GL\tilde{h})_{n}-\frac{k^{2}}{1-k^{2}}\,\tilde{h}_{0}k^{n}=(GR\tilde{h})_{n}-G_{n,0}(R\tilde{h})_{0}+G_{n,0}(L\tilde{h})_{0}-\frac{k^{2}}{1-k^{2}}\,\tilde{h}_{0}k^{n}.
\]
In view of the form of $G$ in (\ref{eq:L-G}) it follows that there
exists $c\in\mathbb{C}$ such that
\[
\tilde{\mathbf{h}}=GR\tilde{\mathbf{h}}+c\hskip0.05em\mathbf{k}
\]
where $\mathbf{\tilde{\mathbf{h}}}$ is a column vector whose entries
are $\tilde{h}_{n}$, $n\geq0$.

$G$ and $R$ can be regarded as matrix operators on $\ell^{2}(\mathbb{N}_{0})$.
As such, it is clear that both of them are bounded, We can even estimate
\[
\|R\|\leq\rho_{N}\,(k+k^{-1}+2).
\]
Choosing $N$ sufficiently large so that $\|GR\|<1$ we conclude that
\[
\tilde{\mathbf{h}}=c\,(I-GR)^{-1}\mathbf{k}\in\ell^{2}(\mathbb{N}_{0}).
\]

This shows uniqueness. In fact, $h_{n}$, $n\geq N$, is prescribed
unambiguously up to a constant multiplier. But (\ref{eq:threetermeq-gen})
along with the assumption $1+x_{n}\neq0$ implies uniqueness for all
$n\geq0$.

(II)~Conversely, with the same choice of $N$ as above, consider
the square summable vector
\[
\tilde{\mathbf{h}}:=(I-GR)^{-1}\mathbf{k}=\mathbf{k}+GR\tilde{\mathbf{h}}.
\]
Note that $(L\mathbf{k})_{n}=0$ for $n\geq1$. Referring again to
(\ref{eq:LG_GL}) we get, for $n\geq1$,
\[
(L\tilde{\mathbf{h}})_{n}=\left(\left(I+\frac{k^{2}}{1-k^{2}}\,\mathbf{e}_{0}\hskip0.03em\mathbf{k}^{T}\right)\!R\tilde{\mathbf{h}}\right)_{\!n}=(R\tilde{\mathbf{h}})_{n}.
\]
This mean that $h_{n}:=\tilde{h}_{n-N}$, $n>N$, satisfies (\ref{eq:threetermeq-gen})
on a neighborhood of $\infty$. Since $1+x_{n}\neq0$ for $n\geq1$,
$(h_{n})_{n>N}$ can be extended to a solution $(h)_{n\geq0}$ of
equation (\ref{eq:threetermeq-gen}) by a descending three-term recurrence.
\end{proof}

\begin{remark} Using the same technique as in the proof of Proposition~\ref{thm:hn_l2}
with slightly stronger assumptions one can obtain, in a routine way,
a more detailed information about the asymptotic behavior of the square
summable solution to (\ref{eq:threetermeq-gen}). No doubt this type
of information is in principle useful but we shall not need it in
the sequel. Nevertheless let us just mention what can be shown but
doing so we omit the proof.

\emph{Let $(h_{n})_{n\geq0}$ be a square summable solution to (\ref{eq:threetermeq-gen})
where again $k\in(0,1)$, and $(s_{n})_{n\geq1}$, $(x_{n})_{n\geq1}$
and $(y_{n})_{n\geq1}$ are complex sequences. If
\[
s_{n}=O(n^{-1}),\ x_{n}=O(n^{-1}),\ y_{n}=O(n^{-1})\ \,\text{as}\ n\to\infty,
\]
then $h_{n}=O(n^{r}k^{n})$ for some $r\geq0$ sufficiently large.
If
\[
s_{n}=O(n^{-1-\epsilon}),\ x_{n}=O(n^{-1-\epsilon}),\ y_{n}=O(n^{-1-\epsilon})\ \,\text{as}\ n\to\infty
\]
for some $\epsilon>0$ then $h_{n}=O(k^{n})$}. \end{remark}

Further we wish to present a quadratic identity for the hypergeometric
functions which will be helpful in the sequel. The identity may be
new. At least we were not able to trace it out in the most common
literature dedicated to the hypergeometric functions.

Let us first recall several identities for contiguous functions \cite[Eq.\,15.2.10]{AbramowitzStegun}
\begin{equation}
(c-a)\,\HG{a-1,b}{c}{z}+\big(2a-c+(b-a)z\big)\,\HG{a,b}{c}{z}+a(z-1)\,\HG{a+1,b}{c}{z}=0\label{eq:contig-a}
\end{equation}
and \cite[Eqs.\,15.2.18,\,15.2.20]{AbramowitzStegun}
\begin{eqnarray}
\HG{a,b+1}{c}{z} & = & \frac{a-c}{b(z-1)}\,\HG{a-1,b}{c}{z}+\frac{c-a-b}{b(z-1)}\,\HG{a,b}{c}{z},\label{eq:contig-b}\\
\noalign{\smallskip}\HG{a,b}{c+1}{z} & = & \frac{c}{(c-b)z}\,\HG{a-1,b}{c}{z}+\frac{c\,(z-1)}{(c-b)z}\,\HG{a,b}{c}{z}.\label{eq:contig-c}
\end{eqnarray}

\begin{proposition} For $a,b,c,d,z\in\mathbb{C}$, $|z|<1$, it holds
true that
\begin{eqnarray}
 &  & \hskip-1.8em(a-c+1)\,\HG{a,b}{c}{z}\HG{a-c+2,b-c+1}{2-c}{z}\nonumber \\
\noalign{\smallskip} &  & -\,a\hskip.01em\,\HG{a+1,b}{c}{z}\HG{a-c+1,b-c+1}{2-c}{z}=\,(1-c)(1-z)^{-a-b+c-1}\nonumber \\
\label{eq:2F1_quadr}
\end{eqnarray}
provided all hypergeometric functions occurring in the expression
are well defined. \end{proposition}

\begin{proof} The verification is very straightforward though rather
tedious. We omit some computational details. Let
\begin{eqnarray*}
F(z) & := & (1-z)^{a+b-c+1}\Bigg(\!(a-c+1)\,\HG{a,b}{c}{z}\HG{a-c+2,b-c+1}{2-c}{z}\\
 &  & \hskip7em-\,a\,\HG{a+1,b}{c}{z}\HG{a-c+1,b-c+1}{2-c}{z}\!\Bigg).
\end{eqnarray*}
We are going to show that $F'(z)=0$. The constant value of $F(z)$
is then determined by putting $z=0$.

To evaluate the derivative one can use the well-known rule
\[
\frac{\text{d}}{\text{d}z}\,\HG{a,b}{c}{z}=\frac{ab}{c}\,\HG{a+1,b+1}{c+1}{z}.
\]
Afterwards we apply (\ref{eq:contig-b}), (\ref{eq:contig-c}) so
that all hypergeometric functions occurring in the resulting expression
have for the second parameter either $b$ or $b-c+1$ and, similarly,
for the third parameter either $c$ or $2-c$. This way we get an
equation of the form
\[
(1-z)^{-a-b+c}\,F'(z)=\sum_{j=0}^{3}A_{j}\,\HG{a-c+j,b-c+1}{2-c}{z}
\]
where $A_{j}$'s are linear combinations of $\,_{2}F_{1}(a+i,b;c;z)$,
$i=-1,0,1,2$, over the field of rational functions in $a$, $b$,
$c$, $z$. Then one can use (\ref{eq:contig-a}) to express the $A_{j}$'s
as linear combinations of $\,_{2}F_{1}(a,b;c;z)$ and $\,_{2}F_{1}(a+1,b;c;z)$
only. Explicitly,
\begin{eqnarray*}
A_{0} & = & -\frac{(a-1)a(a-c+1)}{bz}\,\HG{a+1,b}{c}{z},\\
\noalign{\smallskip}A_{1} & = & \frac{a(a-c+1)(a+b-c+2)}{bz}\,\HG{a,b}{c}{z}\\
\noalign{\smallskip} &  & -\,\frac{a(a-c+1)\big(c-2a+(a-b)z\big)}{bz}\,\HG{a+1,b}{c}{z},\\
\noalign{\smallskip}A_{2} & = & \frac{(a-c+1)(a+b-c+2)\big(c-2a-2+(a-b+1)z\big)}{bz}\,\HG{a,b}{c}{z}\\
\noalign{\smallskip} &  & +\,\frac{a(a-c+1)^{2}(z-1)}{bz}\,\HG{a+1,b}{c}{z},\\
\noalign{\smallskip}A_{3} & = & -\frac{(a-c+1)(a-c+2)(a+b-c+2)(z-1)}{bz}\,\HG{a,b}{c}{z}.
\end{eqnarray*}
The equation can be then rewritten as
\[
(1-z)^{-a-b+c}\,F'(z)=B_{0}\,\HG{a,b}{c}{z}+B_{1}\,\HG{a+1,b}{c}{z}.
\]
But with the aid of (\ref{eq:contig-a}) it can be seen quite straightforwardly
that $B_{0}=B_{1}=0$. \end{proof}

\begin{lemma}\label{thm:hI-hII} Assume that $k\in(0,1)$, $\xi,\eta,\sigma\in\mathbb{C}$
and $\xi-\eta\notin\mathbb{Z}$. Then the sequences $(h_{n}^{(I)})_{n\geq N}$,
$(h_{n}^{(II)})_{n\geq N}$, with
\begin{eqnarray*}
h_{n}^{(I)} & := & k^{n}\,\HG{n+\eta,\omega(\xi,\eta,\sigma)}{\eta-\xi}{1-k^{2}},\\
\noalign{\smallskip}h_{n}^{(II)} & := & \frac{k^{n}\Gamma(n+\xi+1)}{\Gamma(n+\eta)}\text{\,}\HG{n+\xi+1,\omega(\xi,\eta,\sigma)+\xi-\eta+1}{\xi-\eta+2}{1-k^{2}},
\end{eqnarray*}
and
\begin{equation}
\omega(\xi,\eta,\sigma):=\frac{-\xi-k^{2}\eta+(1+k^{2})\sigma}{1-k^{2}}\label{eq:b-x_y_s}
\end{equation}
are well defined for $N\in\mathbb{N}$ sufficiently large and solve
equation (\ref{eq:three-term-eq}) for $n>N$. Moreover,
\begin{equation}
h_{n+1}^{(II)}h_{n}^{(I)}-h_{n+1}^{(I)}h_{n}^{(II)}=\frac{\Gamma(n+\xi+1)}{\Gamma(n+\eta+1)}\,(\xi-\eta+1)k^{-2\xi-2\omega(\xi,\eta,\sigma)-1},\ n\geq N.\label{eq:h_wronsk}
\end{equation}
Hence these solutions are linearly independent. \end{lemma}

\begin{proof} With our assumptions, $h_{n}^{(I)}$ is well defined
for all $n\in\mathbb{Z}$ and $h_{n}^{(II)}$ is well defined for
all $n\in\mathbb{Z}$, $-n-\xi\notin\mathbb{N}$. Furthermore, it
is straightforward to verify with the aid of (\ref{eq:contig-a})
that both $(h_{n}^{(I)})$ and $(h_{n}^{(II)})$ satisfy (\ref{eq:three-term-eq})
for $n>N$. Finally, (\ref{eq:h_wronsk}) is a direct consequence
of (\ref{eq:2F1_quadr}). \end{proof}

Let us recall that
\begin{eqnarray}
 &  & \hskip-1.7em\HG{a,b}{a+b-c+1}{1-z}\!=\,\frac{\Gamma(1+a+b-c)\Gamma(1-c)}{\Gamma(1+a-c)\Gamma(1+b-c)}\,\HG{a,b}{c}{z}\nonumber \\
\noalign{\smallskip} &  & \hskip9em+\,\frac{\Gamma(1+a+b-c)\Gamma(c-1)}{\Gamma(a)\Gamma(b)}\,z^{1-c}\,\HG{a-c+1,b-c+1}{2-c}{z},\nonumber \\
\label{eq:2F1_1-z_z}
\end{eqnarray}
see \cite[Eq.\,15.3.6]{AbramowitzStegun}.

\begin{proposition}\label{thm:hn-plus} Assume that $k\in(0,1)$,
$\xi,\eta,\sigma\in\mathbb{C}$ and $-\xi\notin\mathbb{N}$. Let
\begin{eqnarray}
h_{n}^{(+)} & := & \frac{(1-k^{2})^{-\xi+\eta-1}\,k^{n}\Gamma(n+\xi+1)}{\Gamma\big(n+\omega(\xi,\eta,\sigma)+\xi+1\big)}\,\HG{n+\eta,\omega(\xi,\eta,\sigma)}{n+\xi+\omega(\xi,\eta,\sigma)+1}{k^{2}}\nonumber \\
\noalign{\smallskip} & = & \frac{k^{n}\Gamma(n+\xi+1)}{\Gamma\big(n+\omega(\xi,\eta,\sigma)+\xi+1\big)}\,\HG{n+\xi+1,\omega(\xi,\eta,\sigma)+\xi-\eta+1}{n+\omega(\xi,\eta,\sigma)+\xi+1}{k^{2}},\nonumber \\
\label{eq:hn-plus}
\end{eqnarray}
with $\omega(\xi,\eta,\sigma)$ being defined in (\ref{eq:b-x_y_s}).
Then, up to a constant multiplier, $(h_{n}^{(+)})_{n\geq0}$ is the
unique square summable solution of equation (\ref{eq:three-term-eq}).
Moreover,
\begin{equation}
h_{n}^{(+)}=(1-k^{2})^{-\omega(\xi,\eta,\sigma)-\xi+\eta-1}k^{n}n^{-\omega(\xi,\eta,\sigma)}\!\left(1+O\!\left(\frac{1}{n}\right)\!\right)\ \text{as}\ n\to\infty.\label{eq:hn-plus-asympt}
\end{equation}
\end{proposition}

\begin{proof} The latter equation in (\ref{eq:hn-plus}) follows
from the familiar identity \cite[Eq.\,15.3.3]{AbramowitzStegun}
\[
\HG{a,b}{c}{z}=(1-z)^{c-a-b}\,\HG{c-a,c-b}{c}{z}.
\]
To show that $(h_{n}^{(+)})_{n\geq0}$ is the sought solution we can
make use of solutions $(h_{n}^{(I)})_{n\geq0}$ and $(h_{n}^{(II)})_{n\geq0}$
from Lemma~\ref{thm:hI-hII} which are well defined for all $n\in\mathbb{Z}$
provided $\xi,\xi-\eta\notin\mathbb{Z}$. Then a direct application
of (\ref{eq:2F1_1-z_z}) yields
\[
(1-k^{2})^{\xi-\eta+1}h_{n}^{(+)}=\frac{\Gamma(1+\xi-\eta)}{\Gamma\big(1+\omega(\eta,\xi,\sigma)\big)}\,h_{n}^{(I)}+\frac{(1-k^{2})^{1+\xi-\eta}\Gamma(\eta-\xi-1)}{\Gamma\big(\omega(\xi,\eta,\sigma)\big)}\,h_{n}^{(II)}.
\]
Hence, under these restrictions, $(h_{n}^{(+)})_{n\geq0}$ is also
a solution to (\ref{eq:three-term-eq}). But from (\ref{eq:hn-plus})
it is seen that if $-\xi\notin\mathbb{N}$ then $h_{n}^{(+)}$ is
defined for all $n\in\mathbb{N}_{0}$ and depends continuously on
$\xi$ and $\eta$. Hence the restriction on $\xi$ and $\eta$ can
be relaxed while keeping only the assumption $-\xi\notin\mathbb{N}$.
If so, $h_{n}^{(+)}$ satisfies (\ref{eq:three-term-eq}).

To get the asymptotic expansion we can use Eq.\,15.3.5 in \cite{AbramowitzStegun},
\[
\HG{a,b}{cz}{z}=(1-z)^{-b}\,\HG{b,c-a}{c}{\frac{z}{z-1}},
\]
and Eq.\,15.7.1 ibidem,
\[
\HG{a,b}{c}{z}=1+O\!\left(\frac{1}{|c|}\right)\ \text{as}\ |c|\to\infty,\ \text{with}\ a,b,z\ \text{fixed},
\]
to find that
\[
\HG{n+a,b}{n+c}{z}=(1-z)^{-b}\,\HG{b,c-a}{n+c}{\frac{z}{z-1}}\!=(1-z)^{-b}\!\left(1+O\!\left(\frac{1}{n}\right)\!\right)
\]
as $n\to\infty$. Strictly speaking, this reasoning is applicable
only for $|z|<1/2$ but the asymptotic expansion is known to be valid
also for $|z|<1$, see \cite{KhwajaDaalhuis}. Furthermore, by Stirling's
formula,
\[
\frac{\Gamma(n+a)}{\Gamma(n+b)}=n^{a-b}\!\left(1+O\!\left(\frac{1}{n}\right)\!\right)\ \text{as}\ n\to\infty.
\]
Equation (\ref{eq:hn-plus-asympt}) follows. \end{proof}

\section{General commuting Hankel matrix\label{sec:General-commuting-Hankel}}

Not all Jacobi matrices (\ref{eq:jacobi_mat}) with coefficients of
the form (\ref{eq:alpha-beta-gen}) admit a~nontrivial commuting
Hankel matrix. The goal of the current section is to explore all possible
cases within this class of Jacobi matrices when such a Hankel matrix
exists. The starting point is the following lemma which is proven
in~\cite[Lemma~3]{StampachStovicek_askey-hankel}.

\begin{lemma}\label{thm:Mzw} Let $p$ and $q$ be complex functions
which are meromorphic in a neighborhood of $\infty$ and assume that
the order of the pole at $\infty$ equals $2$ for both of them. Further
let $\epsilon\in\mathbb{C}$, $\epsilon\neq0$, and put, for $z,w\in\mathbb{C}$
sufficiently large,
\[
M(z,w):=\left(\begin{array}{cc}
p(z+\epsilon)-p(w-\epsilon) & q(z+\epsilon)-q(w-\epsilon)\\
p(z-\epsilon)-p(w+\epsilon) & q(z-\epsilon)-q(w+\epsilon)
\end{array}\right)\!.
\]
Let us write the determinant of $M(z,w)$ in the form
\[
\det M(z,w)=\left((z-w)^{2}-4\epsilon^{2}\right)\delta(z,w).
\]
If at least one of the functions $p(z)$ and $q(z)$ is not a polynomial
in $z$ of degree $2$ and the set of functions $\{1,p,q\}$ is linearly
independent, then one of the following two cases happens:

(i) for every $w\in\mathbb{C}$ sufficiently large there exists $\underset{z\to\infty}{\text{lim}}\delta(z,w)\in\mathbb{C}\backslash\{0\}$,

(ii) for every $w\in\mathbb{C}$ sufficiently large there exists $\underset{z\to\infty}{\text{lim}}z\delta(z,w)\in\mathbb{C}\backslash\{0\}$.\\
Consequently, for every $w\in\mathbb{C}$ sufficiently large there
exists $R(w)>0$ such that for all $z\in\mathbb{C}$, $\left|z\right|>R(w)$,
the matrix $M(z,w)$ is regular. \end{lemma}

Consider a semi-infinite Jacobi (tridiagonal) matrix $J$, indexed
by $m,n\in\mathbb{N}_{0}$, which is of the form (\ref{eq:jacobi_mat})
and is determined by the sequences $(\alpha_{n})$ and $(\beta_{n})$
given in (\ref{eq:alpha-beta-gen}), with $k\in(0,1)$, $a,b,c>-1$
and $\sigma\in\mathbb{R}$. Asymptotically we have
\begin{equation}
\alpha_{n}=-n^{2}-(\xi+2)n+A+O\!\left(\frac{1}{n}\right)\ \text{as}\ n\to\infty,\label{eq:alpha-asympt}
\end{equation}
where
\begin{equation}
\xi=\frac{a+b+c}{2}\,.\label{eq:xi}
\end{equation}
and
\begin{equation}
A=a^{2}+b^{2}+c^{2}-2ab-2ac-2bc-4a-4b-4c-8.\label{eq:A}
\end{equation}
Note that $\alpha_{-1}=0$.

Suppose $H$ is a Hankel matrix, $H_{m,n}=h_{m+n}$. Then $H$ and
$J$ commute if and only if it holds true
\begin{equation}
(\alpha_{n}-\alpha_{m})h_{n+m+1}+(\beta_{n}-\beta_{m})h_{n+m}+(\alpha_{n-1}-\alpha_{m-1})h_{n+m-1}=0,\label{eq:a_b_H}
\end{equation}
for all $m,n\geq0$. In particular, letting $m=0$ we have
\[
(\alpha_{n}-\alpha_{0})h_{n+1}+(\beta_{n}-\beta_{0})h_{n}+\alpha_{n-1}h_{n-1}=0\text{ }\ \text{for}\ \text{all}\ n\geq1.
\]
Taking into account the descending recurrence it is clear that, for
any $n\in\mathbb{N}_{0}$,
\[
\text{if}\ \text{ }h_{n}=h_{n+1}=0\text{ }\ \text{then}\ \text{ }h_{0}=h_{1}=\ldots=h_{n}=h_{n+1}=0.
\]

\begin{proposition} Let $\alpha_{n}$, $\beta_{n}$ be the coefficients
given in (\ref{eq:alpha-beta-gen}) and let $J$ be the associated
Jacobi matrix. If there exists a nonzero Hankel matrix commuting with
$J$ then $\alpha_{n}$ depends polynomially on $n$. \end{proposition}

\begin{proof} We will proceed by contradiction. Let us assume that
$\alpha_{n}$ is not a polynomial in $n$ and a nontrivial solution
$(h_{n})_{n\geq0}$ to (\ref{eq:a_b_H}) does exist. Without loss
of generality we can assume that $h_{n}$ is real for all $n$. We
will make use of the fact that $\alpha_{n}$ may be regarded as an
analytic functions in $n$ for $n$ sufficiently large. Of course,
we can make use as well of the fact that $\beta_{n}$ is a polynomial
in $n$. 

Along with (\ref{eq:a_b_H}) we will consider the equation
\begin{equation}
(\alpha_{n-1}-\alpha_{m+1})h_{n+m+1}+(\beta_{n-1}-\beta_{m+1})h_{n+m}+(\alpha_{n-2}-\alpha_{m})h_{n+m-1}=0.\label{eq:Hprime}
\end{equation}
From the asymptotic behavior, as $n\to\infty$,
\[
\alpha_{n}=-n^{2}-(\xi+2)n+O(1),\ \alpha_{n-1}=-n^{2}-\xi n+O(1),
\]
it is obvious that $\{\alpha_{n},\alpha_{n-1},1\}$ is linearly independent
as a set of functions in $n$. Clearly, the same is true for both
$\{\alpha_{n},\beta_{n},1\}$ and $\{\beta_{n},\alpha_{n-1},1\}$
since $\beta_{n}$ is a polynomial in $n$.

Let
\begin{eqnarray*}
\delta_{1}(n,m) & := & \det\!\left(\begin{array}{cc}
\beta_{n}-\beta_{m} & \alpha_{n-1}-\alpha_{m-1}\\
\beta_{n-1}-\beta_{m+1} & \alpha_{n-2}-\alpha_{m}
\end{array}\right)\!,\\
\delta_{2}(n,m) & := & -\det\!\left(\begin{array}{cc}
\alpha_{n}-\alpha_{m} & \alpha_{n-1}-\alpha_{m-1}\\
\alpha_{n-1}-\alpha_{m+1} & \alpha_{n-2}-\alpha_{m}
\end{array}\right)\!,\\
\delta_{3}(n,m) & := & \det\!\left(\begin{array}{cc}
\alpha_{n}-\alpha_{m} & \beta_{n}-\beta_{m}\\
\alpha_{n-1}-\alpha_{m+1} & \beta_{n-1}-\beta_{m+1}
\end{array}\right)\!.
\end{eqnarray*}
According to Lemma~\ref{thm:Mzw}, for all $m$ sufficiently large
there exists $R_{m}\in\mathbb{N}$ such that for all $n\geq R_{m}$,
$\delta_{j}(n,m)\neq0$ for j=1,2,3. Then, by equations (\ref{eq:a_b_H})
and (\ref{eq:Hprime}), the vectors
\[
(h_{n+m+1},h_{n+m},h_{n+m-1})\ \ \text{and}\ \ \big(\delta_{1}(n,m),\delta_{2}(n,m),\delta_{3}(n,m)\big)
\]
are linearly dependent.

Fix sufficiently large $m\in\mathbb{N}_{0}$. Then for all $n\in\mathbb{N}_{0}$,
$n\geq m_{0}:=R_{m}+m$, we have
\[
h_{n+1}=\psi(n)h_{n},\,\ \text{with}\ \text{ }\psi(n):=\frac{\delta_{1}(n-m,m)}{\delta_{2}(n-m,m)}\,.
\]
It is of importance that $\psi(n)$ can be regarded as a meromorphic
function of $n$ in a neighborhood of $\infty$. Particularly, $\psi(n)$
has an asymptotic expansion to all orders as $n\to\infty$.

In view of Lemma~\ref{thm:Mzw} there are only three possible types
of asymptotic behavior of $\psi(n)$ as $n\to0$:
\begin{eqnarray*}
(\text{I}) &  & \psi(n)=\lambda_{1}\!\left(1+O\!\left(\frac{1}{n}\right)\!\right)\!,\\
(\text{II}) &  & \psi(n)=\lambda_{2}\,n\left(1+O\!\left(\frac{1}{n}\right)\!\right)\!,\\
(\text{III}) &  & \psi(n)=\frac{\lambda_{3}}{n}\!\left(1+O\!\left(\frac{1}{n}\right)\!\right)\!,
\end{eqnarray*}
Note that in any case $\lambda_{j}\neq0$. From here one can deduce
the asymptotic behavior of
\[
h_{n}=h_{m_{0}}\,\prod_{k=m_{0}}^{n-1}\psi(k).
\]
In case (I) we have
\[
h_{n}=c_{1}\lambda_{1}^{\,n}\,n^{s_{1}}\left(1+O\!\left(\frac{1}{n}\right)\!\right)\text{ }\text{as}\ n\to\infty,
\]
for some $c_{1},s_{1}\in\mathbb{R}$, $c_{1}\neq0$. In case (II)
we have
\[
h_{n}=c_{2}\lambda_{2}^{\,n}\,n!\,n^{s_{2}}\left(1+O\!\left(\frac{1}{n}\right)\!\right)\text{ }\text{as}\ n\to\infty,
\]
for some $c_{2},s_{2}\in\mathbb{R}$, $c_{2}\neq0$. In case (III)
we have
\[
h_{n}=\frac{c_{3}\lambda_{3}^{\,n}\,n^{s_{3}}}{n!}\left(1+O\!\left(\frac{1}{n}\right)\!\right)\text{ }\text{as}\ n\to\infty,
\]
for some $c_{3},s_{3}\in\mathbb{R}$, $c_{3}\neq0$.

Rewriting (\ref{eq:a_b_H}) and taking into the account the asymptotic
behavior of $\alpha_{n}$ we obtain
\begin{eqnarray}
 &  & \hskip-4emh_{n+1}+\frac{\beta_{n-m}-\beta_{m}}{\alpha_{n-m}-\alpha_{m}}\,h_{n}+\frac{\alpha_{n-m-1}-\alpha_{m-1}}{\alpha_{n-m}-\alpha_{m}}\,h_{n-1}\nonumber \\
 &  & \hskip-4em=\,h_{n+1}-(k+k^{-1})\!\left(1+O\!\left(\frac{1}{n}\right)\!\right)\!h_{n}+\!\left(1+O\!\left(\frac{1}{n}\right)\!\right)\!h_{n-1}=0.\label{eq:H_eq_asympt}
\end{eqnarray}
It is readily seen that the asymptotic behavior of $h_{n}$ of type
(II) and (III) is incompatible with (\ref{eq:H_eq_asympt}). Hence
the only admissible asymptotic behavior of $h_{n}$ is that of type
(I). Without loss of generality we can suppose that $c_{1}=1$. Moreover,
from (\ref{eq:H_eq_asympt}) it is also seen that $\lambda_{1}$ should
solve the equation $\lambda_{1}^{\,2}-(k+k^{-1})\lambda_{1}+1=0$
whence $\lambda_{1}=k$ or $\lambda_{1}=k^{-1}$. For definiteness
let us assume that $\lambda_{1}=k$. The case $\lambda_{1}=k^{-1}$
can be treated analogously. Furthermore, for the sake of simplicity
we will drop the index in $s_{1}$. Thus we obtain
\begin{equation}
h_{n}=k^{n}(n+1)^{s}\varphi(n)\label{eq:H_subst}
\end{equation}
where
\begin{equation}
\varphi(n)=1+\frac{\varphi_{1}}{n}+O\!\left(\frac{1}{n^{2}}\right)\ \ \text{as}\ n\to\infty,\label{eq:phi_asympt}
\end{equation}
with some (undetermined) coefficient $\varphi_{1}$. But in fact,
$\varphi(n)$ has an asymptotic expansion to all orders as $n\to\infty$.

Plugging (\ref{eq:H_subst}) into (\ref{eq:a_b_H}) we obtain
\begin{eqnarray}
 &  & \hskip-3.5em(\alpha_{n}-\alpha_{m})k\!\left(1+\frac{1}{n+m+1}\right)^{\!s}\varphi(n+m+1)+(\beta_{n}-\beta_{m})\varphi(n+m)\nonumber \\
 &  & \hskip5.5em+\,(\alpha_{n-1}-\alpha_{m-1})k^{-1}\!\left(1-\frac{1}{n+m+1}\right)^{\!s}\varphi(n+m-1)=0.\label{eq:H_phi}
\end{eqnarray}

The asymptotic expansion of the LHS of (\ref{eq:H_phi}) as $n\to\infty$,
with $m$ being fixed but otherwise arbitrary, while taking into account
(\ref{eq:phi_asympt}) and (\ref{eq:xi}), yields the expression
\begin{eqnarray*}
 &  & \hskip-2em\frac{(1-k^{2})s-(1+k^{2})(\xi-\sigma)-2k^{2}}{k}\,n+V(k,a,b,c,\sigma,s,\varphi_{1})+(k-k^{-1})sm\\
 &  & \hskip16em-\,\beta_{m}-k\alpha_{m}-k^{-1}\alpha_{m-1}+O\!\left(\frac{1}{n}\right)
\end{eqnarray*}
where $V(k,a,b,c,\sigma,s)$ is a function of the indicated variables
but independent of $m$ and $n$. Necessarily,
\[
s=\frac{(1+k^{2})(\xi-\sigma)+2k^{2}}{1-k^{2}}\,.
\]
Furthermore,
\begin{equation}
\beta_{m}+k\alpha_{m}+k^{-1}\alpha_{m-1}=(k-k^{-1})sm+V(k,a,b,c,\sigma,s,\varphi_{1}),\ \forall m\geq0.\label{eq:beta-alpha-alpha}
\end{equation}

The asymptotic expansion (\ref{eq:alpha-asympt}) can be made more
precise. For $n$ large we have
\[
\alpha_{n}=-n^{2}-(\xi+2)n+A+B(n)
\]
where $A$ is a constant given in (\ref{eq:A}) and
\begin{equation}
B(z)=\sum_{j=1}^{\infty}\frac{b_{j}}{z^{j}}\label{eq:B-z}
\end{equation}
is an analytic function in a neighborhood of $\infty$. Since $\beta_{n}$
is a polynomial in $n$, from (\ref{eq:beta-alpha-alpha}) it is seen
that the coefficients in the asymptotic expansion of
\[
kB(n)+k^{-1}B(n-1),\ \,\text{as}\ n\to\infty,
\]
vanish to all orders. Referring to (\ref{eq:B-z}), from here one
straightforwardly deduces by mathematical induction in $\ell$ that
$b_{j}=0$ for $1\leq j\leq\ell-1$ and all $\ell\geq1$. Whence $B(z)=0$.
In fact, if $b_{j}=0$ for $1\leq j\leq\ell-1$ and some $\ell\in\mathbb{N}$
then
\[
b_{\ell}=\lim_{n\to\infty}n^{\ell}B(n)=\lim_{n\to\infty}(n-1)^{\ell}B(n-1)=\lim_{n\to\infty}n^{\ell}B(n-1)
\]
whence $(k+k^{-1})b_{\ell}=0$.

Thus we conclude that $\alpha_{n}$ is a polynomial in $n$, a contradiction.
\end{proof}

From now on we shall focus on the case when $\alpha_{n}$ is a polynomial
in $n$ while $\beta_{n}$ is the same as in (\ref{eq:alpha-beta-gen}).
Hence
\begin{equation}
\alpha_{n}=-(n+1)(n+a+1),\ \beta_{n}=(k+k^{-1})n\,(n+\sigma).\label{eq:alpha-poly}
\end{equation}
and $\xi$ in (\ref{eq:xi}) simplifies to $\xi=a$\@. Furthermore,
equation (\ref{eq:a_b_H}) reduces to
\begin{equation}
(k+k^{-1})(n+\sigma)h_{n}-(n+a)h_{n-1}-(n+a+2)h_{n+1}=0,\ \,n\geq1.\label{eq:Hn-reduced}
\end{equation}
Referring to Propositions~\ref{thm:hn_l2} and \ref{thm:hn-plus}
we have the following result.

\begin{theorem}\label{thm:sol-H-J} Let $\alpha_{n}$ and $\beta_{n}$
be given by (\ref{eq:alpha-poly}). Denote by $J$ the respective
Jacobi matrix (\ref{eq:jacobi_mat}). Then, up to a constant multiplier,
the only square summable solution to (\ref{eq:Hn-reduced}) reads
\begin{equation}
h_{n}=\frac{k^{n}\Gamma(n+a+1)}{\Gamma\big(n+a+\omega(a,\sigma)+1\big)}\,\HG{n+a+1,\omega(a,\sigma)-1}{n+a+\omega(a,\sigma)+1}{k^{2}},\ n\geq0,\label{eq:sol-H-2F1}
\end{equation}
where
\[
\omega(a,\sigma):=\frac{-2k^{2}+(1+k^{2})(\sigma-a)}{1-k^{2}}\,.
\]
The solution also admits an integral representation,
\begin{equation}
h_{n}=\frac{k^{n}}{\Gamma\big(\omega(a,\sigma)\big)}\,\int_{0}^{1}t^{n+a}\left(\frac{1-t}{1-k^{2}t}\right)^{\!\omega(a,\sigma)-1}\text{d}t.\label{eq:sol-H-int}
\end{equation}
The asymptotic behavior of the solution is as follows,
\[
h_{n}=(1-k^{2})^{-\omega(a,\sigma)+1}k^{n}n^{-\omega(a,\sigma)}\!\left(1+O\!\left(\frac{1}{n}\right)\!\right)\ \text{as}\ n\to\infty.
\]
Then, again up to a constant multiplier, $H_{m,n}=h_{m+n}$, $m,n\in\mathbb{N}_{0}$,
is the only Hankel matrix commuting with $J$ with square summable
columns. Therefore, the matrix $H$ defines an operator on $\mathbb{\ell}^{2}(\mathbb{N}_{0})$
whose domain is the linear hull of the canonical basis. \end{theorem}

\section{The main theorem\label{sec:The-main-theorem}}

Let us introduce four Hankel matrices $H^{(p)}$, $H^{(q)}$, $H^{(r)}$,
$H^{(s)}$, depending on a parameter $k\in(0,1)$,
\begin{equation}
H_{m,n}^{(j)}:=h_{m+n}^{(j)},\ \,\text{for}\ m,n\in\mathbb{N}_{0},\,j=p,q,r,s,\label{eq:Hankel-H-h}
\end{equation}
where
\begin{eqnarray}
\hskip-2emh_{n}^{(p)} & := & \frac{k^{n}\Gamma(n+1/2)}{(n+1)!}\,\HG{n+1/2,1/2}{n+2}{k^{2}}=\,\frac{4k^{n}}{\sqrt{\pi}}\int_{0}^{1}t^{2n}\sqrt{\frac{1-t^{2}}{1-k^{2}t^{2}}}\,\text{d}t,\label{eq:def_hank3}\\
\hskip-2emh_{n}^{(q)} & := & \frac{k^{n}\Gamma(n+3/2)}{(n+1)!}\,\HG{n+3/2,-1/2}{n+2}{k^{2}}=\,\frac{2k^{n}}{\sqrt{\pi}}\int_{0}^{1}t^{2n+2}\sqrt{\frac{1-k^{2}t^{2}}{1-t^{2}}}\,\text{d}t,\label{eq:def_hank4}\\
\hskip-2emh_{n}^{(r)} & := & \frac{k^{n}\Gamma(n+1/2)}{n!}\,\HG{n+1/2,-1/2}{n+1}{k^{2}}=\,\frac{2k^{n}}{\sqrt{\pi}}\int_{0}^{1}t^{2n}\sqrt{\frac{1-k^{2}t^{2}}{1-t^{2}}}\,\text{d}t,\label{eq:def_hank5}\\
\hskip-2emh_{n}^{(s)} & := & \frac{k^{n}\Gamma(n+3/2)}{(n+2)!}\,\HG{n+3/2,1/2}{n+3}{k^{2}}=\,\frac{4k^{n}}{\sqrt{\pi}}\int_{0}^{1}t^{2n+2}\sqrt{\frac{1-t^{2}}{1-k^{2}t^{2}}}\,\text{d}t.\label{eq:def_hank6}
\end{eqnarray}
Note that the latter equality in each row of this array of equations
follows from the integral representation (\ref{eq:gauss_hyp_int_repre}).

Recall definitions of the Stieltjes-Carlitz polynomials $p_{n}(x)$,
$q_{n}(x)$, $r_{n}(x)$ and $s_{n}(x)$ in (\ref{eq:recur_carlitz3}),
(\ref{eq:recur_carlitz4}), (\ref{eq:recur_carlitz5}) and (\ref{eq:recur_carlitz6}),
respectively.

\begin{theorem}\label{thm:hank3456} Each of the Hankel matrices
$H^{(j)}$, $j=p,q,r,s$, represents a positive trace class operator
on $\ell^{2}(\mathbb{N}_{0})$ with simple eigenvalues. The eigenvalues
of $H^{(j)}$,\linebreak{}
$j=p,q,r,s$, if enumerated in descending order, are respectively
\begin{eqnarray*}
{\displaystyle \nu_{m}^{(p)}} & = & \frac{4\sqrt{\pi}}{k}\,\frac{q^{m+1/2}}{1+q^{2m+1}}\,,\ m\geq0,\\
{\displaystyle \nu_{m}^{(q)}} & = & \frac{2\sqrt{\pi}}{k}\,\frac{q^{m+1/2}}{1+q^{2m+1}}\,,\ m\geq0,\\
{\displaystyle \nu_{m}^{(r)}} & = & 2\sqrt{\pi}\,\frac{q^{m}}{1+q^{2m}}\,,\ m\geq0,\\
{\displaystyle \nu_{m}^{(s)}} & = & \frac{4\sqrt{\pi}}{k^{2}}\,\frac{q^{m}}{1+q^{2m}}\,,\ m\geq1.
\end{eqnarray*}
Eigenvector $\Psi_{m}^{(j)}$ corresponding to $\nu_{m}^{(j)}$, $j=p,q,r,s$,
can be chosen with the entries
\begin{eqnarray*}
\left(\Psi_{m}^{(p)}\right)_{n} & := & \frac{(-1)^{n}}{k^{n}(2n)!}\,p_{n}\!\left(\frac{\pi^{2}(2m+1)^{2}}{4K^{2}}\right)\!,\\
\left(\Psi_{m}^{(q)}\right)_{n} & := & \frac{(-1)^{n}}{k^{n}(2n+1)!}\,q_{n}\!\left(\frac{\pi^{2}(2m+1)^{2}}{4K^{2}}\right)\!,\\
\left(\Psi_{m}^{(r)}\right)_{n} & := & \frac{(-1)^{n}}{k^{n}(2n)!}\,r_{n}\!\left(\frac{\pi^{2}m^{2}}{K^{2}}\right)\!,\\
\left(\Psi_{m}^{(s)}\right)_{n} & := & \frac{(-1)^{n}}{k^{n}(2n+1)!}\,s_{n}\!\left(\frac{\pi^{2}m^{2}}{K^{2}}\right)\!,
\end{eqnarray*}
with $n\in\mathbb{N}_{0}$. The $\ell^{2}$-norms of the eigenvectors
equal
\begin{eqnarray*}
\|\Psi_{m}^{(p)}\|^{2} & = & \frac{kK}{2\pi}\,\frac{1+q^{2m+1}}{q^{m+1/2}}\,,\\
\|\Psi_{m}^{(q)}\|^{2} & = & \frac{2kK^{3}}{\pi^{3}}\,\frac{1+q^{2m+1}}{(2m+1)^{2}q^{m+1/2}}\,,\\
\|\Psi_{m}^{(r)}\|^{2} & = & \frac{K}{2\pi}\,\frac{1+q^{2m}}{q^{m}}\,,\\
\|\Psi_{m}^{(s)}\|^{2} & = & \frac{k^{2}K^{3}}{2\pi^{3}}\,\frac{1+q^{2m}}{m^{2}q^{m}}\,.
\end{eqnarray*}
\end{theorem}

\begin{proof} Let us first summarize some general features which
are applicable in each of the four cases. As is explicitly indicated
below, each sequence $(h_{n}^{(j)})$, $j=p,q,r,s$, coincides with
a solution $h_{n}$ given in (\ref{eq:sol-H-2F1}) for a particular
choice of parameters $a>-1$ and $\sigma>0$, and in each case we
have $\omega(a,\sigma)>0$. From the integral representation (\ref{eq:sol-H-int})
it is then clear that $h_{n}^{(j)}>0$ for $n\geq0$ and
\[
\sum_{m,n=0}^{\infty}h_{m+n}^{(j)}=\frac{1}{\Gamma\big(\omega(a,\sigma)\big)}\,\int_{0}^{1}\frac{t^{a}}{(1-kt)^{2}}\left(\frac{1-t}{1-k^{2}t}\right)^{\!\omega(a,\sigma)-1}\text{d}t<\infty.
\]
Consequently, each Hankel matrix $H^{(j)}$ represents a trace class
operator on $\ell^{2}(\mathbb{N}_{0})$.

Furthermore, as can be deduced from Theorem~\ref{thm:sol-H-J}, each
Hankel matrix $H^{(j)}$ commutes with a certain Jacobi matrix $J=J^{(j)}$
of the form (\ref{eq:jacobi_mat}) where
\begin{equation}
\alpha_{n}=-k\,(n+1)(n+a+1),\ \beta_{n}=k\,(k+k^{-1})n\,(n+\sigma)+d,\label{eq:alpha-beta-symm}
\end{equation}
and $J^{(j)}$ is therefore determined by a proper choice of the parameters
$a$ and $\sigma$ (note that the multiplicative constant $k$ and
the additive constant $d$ are inessential for the commutation relation).
Jacobi matrix $J^{(j)}$ turns out to have a pure point spectrum which
is necessarily simple. In addition, the Jacobi matrix in each case
is known to determine a unique self-adjoint operator on $\ell^{2}(\mathbb{N}_{0})$
(determinate case). Then from the formal commutation relation between
$H^{(j)}$ and $J^{(j)}$ on the level of semi-infinite matrices as
well as from the fact that the matrix $H^{(j)}$ corresponds to a
bounded operator it readily follows that $H^{(j)}$ preserves the
domain of $J^{(j)}$ (provided $H^{(j)}$ and $J^{(j)}$ are both
regarded as operators). From the simplicity of the spectrum of $J^{(j)}$
one deduces that every eigenvector of $J^{(j)}$ is at the same time
an eigenvector of $H^{(j)}$. Hence a diagonalization of $J^{(j)}$
provides, too, a diagonalization of $H^{(j)}$.

This is also a familiar fact that if $\lambda$ is an eigenvalue of
$J^{(j)}$ then for a corresponding eigenvector one can choose the
vector
\[
\Psi(\lambda):=\big(\hat{P}_{0}(\lambda),\hat{P}_{1}(\lambda),\hat{P}_{2}(\lambda),\ldots\big)^{T}
\]
where $\big(\hat{P}_{n}(x)\big)_{n\geq0}$ is the orthonormal polynomial
sequence associated with $J^{(j)}$ which is unambiguously determined
by letting $\hat{P}_{0}(x)=1$. If
\[
\Spec_{p}J^{(j)}=\{\lambda_{m};\ m\geq0\}
\]
is the point spectrum of $J^{(j)}$ and $\mu$ is the respective orthogonality
measure (normalized as a probability measure and supported on $\Spec_{p}J^{(j)}$)
then the orthogonality relation reads
\[
\sum_{\ell=0}^{\infty}\mu_{\ell}\hat{P}_{m}(\lambda_{\ell})\hat{P}_{n}(\lambda_{\ell})=\delta_{m,n},
\]
where $\mu_{\ell}:=\mu(\{\lambda_{\ell}\})$, and dually,
\[
\sum_{m=0}^{\infty}\hat{P}_{m}(\lambda_{\ell})\hat{P}_{m}(\lambda_{r})=\frac{1}{\mu_{\ell}}\,\delta_{\ell,r}.
\]
Hence
\[
\|\Psi(\lambda_{\ell})\|^{2}=\frac{1}{\mu_{\ell}}\,.
\]

Denoting by $\{\mathbf{e}_{n};\,n\geq0\}$ the canonical basis in
$\ell^{2}(\mathbb{N}_{0})$ we obtain a unitary mapping
\begin{equation}
U\!:\ell^{2}(\mathbb{N}_{0})\to L^{2}\left((0,\infty),\text{d}\mu\right):\mathbf{e}_{n}\mapsto\hat{P}_{n},\label{eq:U}
\end{equation}
which diagonalizes both $J^{(j)}$ and $H^{(j)}$. The operator $UH^{(j)}U^{-1}$
is a multiplication operator by a function $h^{(j)}(x)$ which obeys
\begin{eqnarray}
\hskip-2emh^{(j)}(x) & = & h^{(j)}(x)\hat{P}_{0}(x)\,=\,\sum_{n=0}^{\infty}H_{n,0}^{(j)}\hat{P}_{n}(x).\label{eq:mult-by-h}
\end{eqnarray}
Then the values $h^{(j)}(\lambda_{m})$, $m\geq0$, are exactly the
eigenvalues of $H^{(j)}$.

(i)~~The sequence $h_{n}^{(p)}$, $n\geq0$, in (\ref{eq:def_hank3})
coincides with the solution $h_{n}$ given in (\ref{eq:sol-H-2F1})
for the values of parameters
\[
a=-\frac{1}{2}\,,\ \sigma=\frac{1}{1+k^{2}}\,,\ \text{whence}\ \,\omega(a,\sigma)=\frac{3}{2}\,.
\]
Theorem~\ref{thm:sol-H-J} then implies that the Hankel matrix $H^{(p)}$
commutes with a Jacobi matrix $J^{(p)}$ of the form (\ref{eq:jacobi_mat})
where $\alpha_{n}$ and $\beta_{n}$ are replaced by
\[
\alpha_{n}^{(p)}:=-2k(n+1)(2n+1)\ \,\text{and}\ \,\beta_{n}^{(p)}:=k^{2}(2n)^{2}+(2n+1)^{2}.
\]
The recurrence (\ref{eq:recur_carlitz3}) means that
\[
p_{n+1}(x)=(x-\beta_{n}^{(p)})\,p_{n}(x)-\left(\alpha_{n-1}^{(p)}\right)^{2}p_{n-1}(x),\ n\geq0.
\]
Hence the Jacobi matrix $J^{(p)}$ corresponds to the Family~\#3
of the Stieltjes-Carlitz polynomials.

From (\ref{eq:og_rel_carlitz3}) we know that the set of functions
\[
\hat{P}_{n}(x):=\frac{(-1)^{n}p_{n}(x)}{k^{n}(2n)!}\,,\ n\in\mathbb{N}_{0},
\]
is an orthonormal basis of the Hilbert space $L^{2}\left((0,\infty),\text{d}\mu\right)$
where $\mu$ is given in (\ref{eq:mu_carlitz34}). Referring to (\ref{eq:U})
and (\ref{eq:mult-by-h}), the diagonalized operator $UH^{(p)}U^{-1}$
is a multiplication operator by a function $h^{(p)}(x)$ which can
be computed as follows (see (\ref{eq:Hankel-H-h}) and (\ref{eq:def_hank3}))
\[
h^{(p)}(x)=\sum_{n=0}^{\infty}(-1)^{n}\frac{h_{n}^{(p)}p_{n}(x)}{k^{n}(2n)!}=\frac{4}{\sqrt{\pi}}\sum_{n=0}^{\infty}(-1)^{n}\frac{p_{n}(x)}{(2n)!}\int_{0}^{1}t^{2n}\sqrt{\frac{1-t^{2}}{1-k^{2}t^{2}}}\,\text{d}t.
\]
From (\ref{eq:h3_sum_form}) we find that
\[
h^{(p)}(x)=\frac{4}{\sqrt{\pi}}\int_{0}^{K}\cos(\sqrt{x}u)\cn(u)\,du.
\]

Finally, the last equation combined with the Fourier series (\ref{eq:four_cn})
allows us to evaluate the function $h^{(p)}(x)$ at the spectral points
$\lambda_{m}$ of the Jacobi matrix $J^{(p)}$, as given in (\ref{eq:mu_carlitz34}).
The obtained values read
\[
h^{(p)}(\lambda_{m})=\frac{4\sqrt{\pi}}{k}\frac{q^{m+1/2}}{1+q^{2m+1}}\,,\ \,m\geq0,
\]
and these are in fact the eigenvalues of $H^{(p)}$.

(ii)~~The sequence $h_{n}^{(q)}$, $n\geq0$, in (\ref{eq:def_hank4})
coincides with the solution $h_{n}$ given in (\ref{eq:sol-H-2F1})
for the values of parameters
\[
a=\frac{1}{2}\,,\ \sigma=\frac{1+2k^{2}}{1+k^{2}}\,,\ \text{whence}\ \,\omega(a,\sigma)=\frac{1}{2}\,.
\]
Theorem~\ref{thm:sol-H-J} then implies that the Hankel matrix $H^{(q)}$
commutes with a Jacobi matrix $J^{(q)}$ of the form (\ref{eq:jacobi_mat})
where $\alpha_{n}$ and $\beta_{n}$ are replaced by
\[
\alpha_{n}^{(q)}:=-2k(n+1)(2n+3)\ \,\text{and}\ \,\beta_{n}^{(q)}:=(2n+1)^{2}+k^{2}(2n+2)^{2}.
\]
The recurrence (\ref{eq:recur_carlitz4}) means that
\[
q_{n+1}(x)=(x-\beta_{n}^{(q)})\,q_{n}(x)-\left(\alpha_{n-1}^{(q)}\right)^{2}q_{n-1}(x),\ n\geq0.
\]
Hence the Jacobi matrix $J^{(q)}$ corresponds to the Family~\#4
of the Stieltjes-Carlitz polynomials.

From (\ref{eq:og_rel_carlitz4}) we know that the set of functions
\[
\hat{Q}_{n}(x):=\frac{(-1)^{n}q_{n}(x)}{k^{n}(2n+1)!}\,,\ n\geq0,
\]
is an orthonormal basis of the Hilbert space $L^{2}\left((0,\infty),x\text{d}\mu(x)\right)$
where $\mu$ is given by (\ref{eq:mu_carlitz34}). Referring to (\ref{eq:U})
and (\ref{eq:mult-by-h}), with $\hat{Q}_{n}$ instead of $\hat{P}_{n}$,
the diagonalized operator $UH^{(q)}U^{-1}$ is a multiplication operator
by a function $h^{(q)}(x)$ which can be computed as follows (see
(\ref{eq:Hankel-H-h}) and (\ref{eq:def_hank4}))
\[
h^{(q)}(x)=\sum_{n=0}^{\infty}(-1)^{n}\frac{h_{n}^{(q)}q_{n}(x)}{k^{n}(2n+1)!}=\frac{2}{\sqrt{\pi}}\sum_{n=0}^{\infty}(-1)^{n}\frac{q_{n}(x)}{(2n+1)!}\int_{0}^{1}t^{2n+2}\sqrt{\frac{1-k^{2}t^{2}}{1-t^{2}}}\,\text{d}t.
\]
From (\ref{eq:h4_sum_form}) we find that
\[
h^{(q)}(x)=\frac{2}{\sqrt{\pi}}\int_{0}^{K}\cos(\sqrt{x}u)\cn(u)\,\text{d}u.
\]

Finally, the last equation combined with the Fourier series (\ref{eq:four_cn})
allows us to evaluate the function $h^{(q)}(x)$ at the spectral points
$\lambda_{m}$ of the Jacobi matrix $J^{(q)}$, as given in (\ref{eq:mu_carlitz34}).
The obtained values read
\[
h^{(q)}(\lambda_{m})=\frac{2\sqrt{\pi}}{k}\frac{q^{m+1/2}}{1+q^{2m+1}}\,,\ \,m\geq0,
\]
and these are in fact the eigenvalues of $H^{(q)}$.

(iii)~~The sequence $h_{n}^{(r)}$, $n\geq0$, in (\ref{eq:def_hank5})
coincides with the solution $h_{n}$ given in (\ref{eq:sol-H-2F1})
for the values of parameters
\[
a=-\frac{1}{2}\,,\ \sigma=\frac{k^{2}}{1+k^{2}}\,,\ \text{whence}\ \,\omega(a,\sigma)=\frac{1}{2}\,.
\]
Theorem~\ref{thm:sol-H-J} then implies that the Hankel matrix $H^{(r)}$
commutes with a Jacobi matrix $J^{(r)}$ of the form (\ref{eq:jacobi_mat})
where $\alpha_{n}$ and $\beta_{n}$ are replaced by
\[
\alpha_{n}^{(r)}:=-2k(n+1)(2n+1)\ \,\text{and}\ \,\beta_{n}^{(r)}:=(2n)^{2}+k^{2}(2n+1)^{2}.
\]
The recurrence (\ref{eq:recur_carlitz5}) means that
\[
r_{n+1}(x)=(x-\beta_{n}^{(r)})\,r_{n}(x)-\left(\alpha_{n-1}^{(r)}\right)^{2}r_{n-1}(x),\ n\geq0.
\]
Hence the Jacobi matrix $J^{(r)}$ corresponds to the Family~\#5
of the Stieltjes-Carlitz polynomials.

From (\ref{eq:og_rel_carlitz5}) we know that the set of functions
\[
\hat{R}_{n}(x):=\frac{(-1)^{n}r_{n}(x)}{k^{n}(2n)!}\,,\ n\geq0,
\]
is an orthonormal basis of the Hilbert space $L^{2}\left((0,\infty),\text{d}\mu(x)\right)$
where $\mu$ is given by (\ref{eq:mu_carlitz56}). Referring to (\ref{eq:U})
and (\ref{eq:mult-by-h}), with $\hat{R}_{n}$ instead of $\hat{P}_{n}$,
the diagonalized operator $UH^{(r)}U^{-1}$ is a multiplication operator
by a function $h^{(r)}(x)$ which can be computed as follows (see
(\ref{eq:Hankel-H-h}) and (\ref{eq:def_hank5}))

\[
h^{(r)}(x)=\sum_{n=0}^{\infty}(-1)^{n}\frac{h_{n}^{(r)}r_{n}(x)}{k^{n}(2n)!}=\frac{2}{\sqrt{\pi}}\sum_{n=0}^{\infty}(-1)^{n}\frac{r_{n}(x)}{(2n)!}\int_{0}^{1}t^{2n}\sqrt{\frac{1-k^{2}t^{2}}{1-t^{2}}}\,\text{d}t.
\]
From (\ref{eq:h5_sum_form}) we find that
\[
h^{(r)}(x)=\frac{2}{\sqrt{\pi}}\int_{0}^{K}\cos(\sqrt{x}u)\dn(u)\,\text{d}u.
\]

Finally, the last equation combined with the Fourier series (\ref{eq:four_dn})
allows us to evaluate the function $h^{(r)}(x)$ at the spectral points
$\lambda_{m}$ of the Jacobi matrix $J^{(r)}$, as given in (\ref{eq:mu_carlitz56}).
The obtained values read
\[
h^{(r)}(\lambda_{m})=\frac{2\sqrt{\pi}q^{m}}{1+q^{2m}}\,,\ \,m\geq0,
\]
and these are in fact the eigenvalues of $H^{(r)}$.

(iv)~~The sequence $h_{n}^{(s)}$, $n\geq0$, in (\ref{eq:def_hank6})
coincides with the solution $h_{n}$ given in (\ref{eq:sol-H-2F1})
for the values of parameters
\[
a=\frac{1}{2}\,,\ \sigma=\frac{2+k^{2}}{1+k^{2}}\,,\ \text{whence}\ \,\omega(a,\sigma)=\frac{3}{2}\,.
\]
Theorem~\ref{thm:sol-H-J} then implies that the Hankel matrix $H^{(s)}$
commutes with a Jacobi matrix $J^{(s)}$ of the form (\ref{eq:jacobi_mat})
where $\alpha_{n}$ and $\beta_{n}$ are replaced by
\[
\alpha_{n}^{(s)}:=-2k(n+1)(2n+3)\ \,\text{and}\ \,\beta_{n}^{(s)}:=k^{2}(2n+1)^{2}+(2n+2)^{2}.
\]
The recurrence (\ref{eq:recur_carlitz6}) means that
\[
s_{n+1}(x)=(x-\beta_{n}^{(s)})\,s_{n}(x)-\left(\alpha_{n-1}^{(s)}\right)^{2}s_{n-1}(x),\ n\geq0.
\]
Hence the Jacobi matrix $J^{(s)}$ corresponds to the Family~\#6
of the Stieltjes-Carlitz polynomials.

From (\ref{eq:og_rel_carlitz6}) we know that the set of functions
\[
\hat{S}_{n}(x):=\frac{(-1)^{n}s_{n}(x)}{k^{n}(2n+1)!}\,,\ n\geq0,
\]
is an orthonormal basis of the Hilbert space $L^{2}\left((0,\infty),k^{-2}x\text{d}\mu(x)\right)$
where $\mu$ is given in (\ref{eq:mu_carlitz56}). Referring to (\ref{eq:U})
and (\ref{eq:mult-by-h}), with $\hat{S}_{n}$ instead of $\hat{P}_{n}$,
the diagonalized operator $UH^{(s)}U^{-1}$ is a multiplication operator
by a function $h^{(s)}(x)$ which can be computed as follows (see
(\ref{eq:Hankel-H-h}) and (\ref{eq:def_hank6}))
\[
h^{(s)}(x)=\sum_{n=0}^{\infty}(-1)^{n}\frac{h_{n}^{(s)}s_{n}(x)}{k^{n}(2n+1)!}=\frac{4}{\sqrt{\pi}}\sum_{n=0}^{\infty}(-1)^{n}\frac{s_{n}(x)}{(2n+1)!}\int_{0}^{1}t^{2n+2}\sqrt{\frac{1-t^{2}}{1-k^{2}t^{2}}}\,\text{d}t.
\]
From (\ref{eq:h6_sum_form}) we find that
\[
h^{(s)}(x)=-\frac{4\sqrt{1-k^{2}}\sin(\sqrt{x}K)}{k^{2}\sqrt{\pi x}}+\frac{4}{k^{2}\sqrt{\pi}}\int_{0}^{K}\cos(\sqrt{x}u)\dn(u)\,\text{d}u.
\]

Finally, the last equation combined with the Fourier series (\ref{eq:four_dn})
allows us to evaluate the function $h^{(s)}(x)$ at the spectral points
$\lambda_{m}$ of the Jacobi matrix $J^{(s)}$, as given in (\ref{eq:mu_carlitz56}).
The obtained values read
\[
h^{(s)}(\lambda_{m})=\frac{4\sqrt{\pi}}{k^{2}}\,\frac{q^{m}}{1+q^{2m}}\,,\ \,m\geq1,
\]
and these are in fact the eigenvalues of $H^{(s)}$. Note that $\lambda_{0}=0$
is not an eigenvalue of $J^{(s)}$ since it does not belong to the
support of the measure of orthogonality in (\ref{eq:og_rel_carlitz6}).
Consequently, $h^{(s)}(0)$ is not an eigenvalue of $H^{(s)}$. \end{proof}

\section{Families \#1 and \#2, a generalization to weighted Hankel matrices\label{sec:Families-1-2}}

The main focus of the paper so far was on Hankel matrices which admit
an explicit solution of the spectral problem owing to their close
relationship to the Stieltjes-Carlitz polynomials. This concerns Families
\#3, \#4, \#5 and \#6 only. Our approach does not lead to explicitly
diagonalizable Hankel matrices in case of Families \#1 and \#2. In
this section we propose an extension of the forgoing results by considering
also weighted Hankel matrices of the form
\begin{equation}
H_{m,n}=w_{m}w_{n}h_{m+n},\ m,n\in\mathbb{N}_{0}.\label{eq:def_whank}
\end{equation}
The weights $w_{n}$ are supposed to be positive. Admitting nontrivial
(non-constant) weights we are able to enrich the list of explicitly
diagonalizable matrix operators by several additional items and, in
particular, the generalized approach can be applied to Families \#1
and \#2 as well. Apart of this generalization the basic scheme remains
practically the same as in Section~\ref{sec:The-main-theorem}. This
is why we try to be rather brief in the current section and we omit
some details for routine steps in the derivations to follow.

First of all we have to modify equation (\ref{eq:a_b_H}). The formal
commutation relation $HJ=JH$ between a weighted Hankel matrix $H$
and a Jacobi matrix of the form (\ref{eq:jacobi_mat}), where the
multiplication is understood on the level of semi-infinite matrices,
is satisfied if and only if
\begin{eqnarray*}
 &  & (\beta_{m}-\beta_{n})w_{m}w_{n}h_{m+n}+(\alpha_{m-1}w_{m-1}w_{n}-\alpha_{n-1}w_{m}w_{n-1})h_{m+n-1}\\
 &  & +\,(\alpha_{m}w_{m+1}w_{n}-\alpha_{n}w_{m}w_{n+1})h_{m+n+1}=0
\end{eqnarray*}
holds for all $m,n\in\mathbb{N}_{0}$, $m<n$. Here and everywhere
in what follows we assume that $\alpha_{-1}=0$. An easy computation
leads to the following lemma.

\begin{lemma}\label{lem:diff_eq} Let $J$ be a Jacobi matrix (\ref{eq:jacobi_mat})
with entries given by (\ref{eq:alpha-beta-gen}). Then a matrix $H$
with entries (\ref{eq:def_whank}) commutes formally, on the level
of semi-infinite matrices, with $J$ provided the sequence $(h_{n})_{n\geq0}$
satisfies the difference equation
\begin{equation}
(k+k^{-1})(n+\sigma)h_{n}-(n+a)h_{n-1}-(n+b+c+2)h_{n+1}=0,\ n\geq1,\label{eq:diff_h}
\end{equation}
and
\begin{equation}
w_{n}=\sqrt{\frac{(b+1)_{n}(c+1)_{n}}{n!\,(a+1)_{n}}}\,,\ n\geq0.\label{eq:weight_gener}
\end{equation}
\end{lemma}

Note that equation (\ref{eq:diff_h}) is again of type (\ref{eq:three-term-eq})
which has been studied in Section~\ref{sec:A-three-term-recurrence}.

Similarly as in Section~\ref{sec:Auxiliary-results} we shall need
some auxiliary results. All of them can be derived in a routine way
by using standard methods. The following proposition was shown in
\cite[Thm.~3.1]{ismail-valent_ijm98} and in the course of the proof
of Theorem~3.3.1 in \cite{ismail-etal_jat01} though the notation
therein was different from ours.

\begin{proposition}\label{thm:asympt-SC-poly-fg} The leading terms
in the asymptotic expansion of the Stieltjes-Carlitz polynomials $f_{n}(x)$
and $g_{n}(x)$ defined in (\ref{eq:recur_carlitz1}) and (\ref{eq:recur_carlitz2}),
respectively, are as follows
\begin{eqnarray*}
\frac{f_{n}(-x)}{(2n)!} & = & \frac{1}{\sqrt{\pi n(1-k^{2})}}\,\cos\!\big(\sqrt{x}K\big)+o\!\left(\frac{1}{n}\right)\!,\\
\frac{g_{n}(-x)}{(2n+1)!} & = & \frac{1}{\sqrt{\pi n(1-k^{2})}}\,\frac{\sin\!\big(\sqrt{x}K\big)}{\sqrt{x}}+o\!\left(\frac{1}{n}\right)\!,
\end{eqnarray*}
 as $n\to\infty$. Here $x$ is an arbitrary fixed complex number.
\end{proposition}

\begin{proposition}\label{thm:asympt-fg} For $x\in\mathbb{C}$ and
the Stieltjes-Carlitz polynomials $f_{n}(x)$ and $g_{n}(x)$, defined
in (\ref{eq:recur_carlitz1}), and (\ref{eq:recur_carlitz2}), respectively,
it holds true that
\begin{eqnarray}
\frac{\pi}{4}\sum_{n=0}^{\infty}\frac{f_{n}(-x)}{4^{n}n!(n+1)!} & = & \int_{0}^{K}\cos(\sqrt{x}u)\cn(u)\,\text{d}u,\label{eq:h1_sum_form-1}\\
\frac{\pi}{8}\sum_{n=0}^{\infty}\frac{g_{n}(-x)}{4^{n}n!(n+2)!} & = & -\frac{\sqrt{1-k^{2}}}{k^{2}}\,\frac{\sin(\sqrt{x}K)}{\sqrt{x}}+\frac{1}{k^{2}}\int_{0}^{K}\cos(\sqrt{x}u)\dn(u)\,\text{d}u.\nonumber \\
\label{eq:h2_sum_form-1}
\end{eqnarray}
\end{proposition}

\begin{proof} Equation (\ref{eq:h1_sum_form-1}) can be derived from
the formula for the generating function  (\ref{eq:gener_func_carlitz1}).
One has to write $-x$ instead of $x$ and differentiate the formula
term-wise with respect to $u$ and use the integral identity
\[
\int_{0}^{K}\sn^{2n}(u)\cn(u)\,\frac{\text{d}\sn(u)}{\text{d}u}\,\text{d}u=\int_{0}^{1}y^{2n}\sqrt{1-y^{2}}\,\text{d}y=\frac{\pi}{2^{2n+2}}\frac{(2n)!}{n!(n+1)!},\ \,n\geq0.
\]

Very similarly, equation (\ref{eq:h2_sum_form-1}) can be derived
from the formula for the generating function  (\ref{eq:gener_func_carlitz2}).
This time the integral identity
\[
\int_{0}^{K}\sn^{2n}(u)\cn^{3}(u)\,\frac{\text{d}\sn(u)}{\text{d}u}\,\text{d}u=\int_{0}^{1}y^{2n}\left(1-y^{2}\right)^{3/2}\,\text{d}y=\frac{3\pi}{2^{2n+3}}\frac{(2n)!}{n!(n+2)!},\ \,n\geq0,
\]
turns out to be useful. This case is slightly more complicated since
finally one has to integrate by parts on the RHS to get the desired
expression.

Manipulations used during the derivation in both cases can be justified
with the aid of Proposition~\ref{thm:asympt-fg}. \end{proof}

Below we present two weighted Hankel matrices which have comparatively
simple form and which are related to Families \#1 and \#2 of the Stieltjes-Carlitz
polynomials,
\begin{eqnarray}
H_{m,n}^{(f)} & := & \binom{2n}{n}\binom{2m}{m}\left(\frac{k}{4}\right)^{\!m+n}\,\frac{\sqrt{(2n+1)(2m+1)}}{m+n+1}\,,\label{eq:def_hank1}\\
H_{m,n}^{(g)} & := & \binom{2n+1}{n}\binom{2m+1}{m}\left(\frac{k}{4}\right)^{\!m+n}\,\frac{\sqrt{(m+1)(n+1)}}{m+n+2}\,,\label{eq:def_hank2}
\end{eqnarray}
$m,n\in\mathbb{N}_{0}$.

\begin{theorem}\label{thm:hank12} The weighted Hankel matrices $H^{(f)}$
and $H^{(g)}$ represent both positive trace class operators on $\ell^{2}(\mathbb{N}_{0})$
with simple eigenvalues. We have:

\noindent (i)~~Eigenvalues of $H^{(f)}$ enumerated in descending
order are
\[
{\displaystyle \nu_{m}^{(f)}=\frac{4}{k}\,\frac{q^{m+1/2}}{1+q^{2m+1}}}\,,\ m\geq0.
\]
\noindent An eigenvector $\Psi_{m}^{(f)}$ corresponding to $\nu_{m}^{(f)}$
can be chosen with the entries
\[
\left(\Psi_{m}^{(f)}\right)_{n}=\frac{1}{k^{n}(2n)!\sqrt{2n+1}}\,f_{n}\!\left(-\frac{\pi^{2}(2m+1)^{2}}{4K^{2}}\right)\!,\ \,n\geq0,
\]
and its $\ell^{2}$-norm equals
\[
\|\Psi_{m}^{(f)}\|^{2}=\frac{kK^{2}}{\pi^{2}}\,\frac{1-q^{2m+1}}{(2m+1)q^{m+1/2}}\,.
\]

\noindent (ii)~~Eigenvalues of $H^{(g)}$ enumerated in descending
order are
\[
{\displaystyle \nu_{m}^{(g)}=\frac{8}{k^{2}}\,\frac{q^{m}}{1+q^{2m}}}\,,\ m\geq1.
\]
\noindent An eigenvector $\Psi_{m}^{(g)}$ corresponding to $\nu_{m}^{(g)}$
can be chosen with the entries
\[
\left(\Psi_{m}^{(g)}\right)_{n}=\frac{1}{k^{n}(2n+1)!\,\sqrt{n+1}}\,g_{n}\!\left(-\frac{\pi^{2}m^{2}}{K^{2}}\right)\!,\ \,n\geq0,
\]
and its $\ell^{2}$-norm equals
\[
\|\Psi_{m}^{(g)}\|^{2}=\frac{k^{2}K^{4}}{\pi^{4}}\,\frac{1-q^{2m}}{m^{3}q^{m}}\,.
\]
\end{theorem}

\begin{proof} The basic scheme remains literally the same as explained
in the introductory part of the proof of Theorem~\ref{thm:hank3456}.
Comparing (\ref{eq:diff_h}) to (\ref{eq:three-term-eq}) we let,
in the latter equation,
\[
\xi=a,\ \eta=b+c+2.
\]
Recall also definition (\ref{eq:b-x_y_s}) of $\omega(\xi,\eta,\sigma)$.

(i)~~Consider the entries $\alpha_{n}$, $\beta_{n}$, as given
in (\ref{eq:alpha-beta-gen}), for the values of parameters
\[
a=0,\;b=1/2,\;c=-1/2,\;\sigma=1,\ \mbox{whence}\ \,\xi=0,\;\eta=2,\ \omega(\xi,\eta,\sigma)=1.
\]
Then the weight given in (\ref{eq:weight_gener}) equals
\[
w_{n}^{(f)}=\sqrt{\frac{\left(\frac{3}{2}\right)_{\!n}\left(\frac{1}{2}\right)_{\!n}}{n!\,(1)_{n}}}=\frac{\sqrt{2n+1}}{2^{2n}}\,\binom{2n}{n}.
\]
Referring to Proposition~\ref{thm:hn-plus}, we can choose for $h_{n}^{(f)}$
the square summable solution (\ref{eq:hn-plus}) of equation (\ref{eq:three-term-eq}),
\[
h_{n}^{(f)}=\frac{k^{n}\Gamma(n+1)}{\Gamma\big(n+2\big)}\,\HG{n+1,0}{n+2}{k^{2}}=\frac{k^{n}}{n+1}\,.
\]
One can check that $H_{m,n}^{(f)}=w_{m}^{(f)}w_{n}^{(f)}h_{m+n}^{(f)}$
coincides with (\ref{eq:def_hank1}). By Lemma~\ref{lem:diff_eq},
the weighted Hankel matrix $H^{(f)}$ commutes with $J=J^{(f)}$ introduced
in (\ref{eq:jacobi_mat}) where we put $\alpha_{n}=\alpha_{n}^{(f)}$,
$\beta_{n}=\beta_{n}^{(f)}$,
\[
\alpha_{n}^{(f)}:=-2k(n+1)\sqrt{(2n+1)(2n+3)}\,,\ \beta_{n}^{(f)}:=(k^{2}+1)(2n+1)^{2}.
\]
From (\ref{eq:recur_carlitz1}) it is seen that the Jacobi matrix
$J^{(f)}$ corresponds to the Family~\#1 of the Stieltjes-Carlitz
polynomials.

From the asymptotic expansion
\[
\binom{2n}{n}=\frac{4^{n}}{\sqrt{\pi n}}\left(1+O\!\left(\frac{1}{n}\right)\right)\ \text{as}\ n\to\infty
\]
it seen that $\sum_{m,n=0}^{\infty}H_{m,n}^{(f)}<\infty$ and therefore
the matrix $(H_{m,n}^{(f)})$ determines a trace class operator on
$\ell^{2}(\mathbb{N}_{0})$.

By the orthogonality relation (\ref{eq:og_rel_carlitz1}), the functions
\[
\hat{F}_{n}(x):=\frac{f_{n}(-x)}{k^{n}(2n)!\sqrt{2n+1}}\,,\ \,n\in\mathbb{N}_{0},
\]
form an orthonormal basis of the Hilbert space $L^{2}\left((0,\infty),\text{d}\mu\right)$
where $\mu$ is defined in (\ref{eq:mu_carlitz1}). With this orthonormal
basis on hand and with the canonical basis in $\ell^{2}(\mathbb{N}_{0})$
we can construct a unitary mapping $U$ which diagonalizes $J^{(f)}$
and, at the same time, $H^{(f)}$. $UH^{(f)}U^{-1}$ becomes a multiplication
operator by a function $h^{(f)}(x)$ on $L^{2}\left((0,\infty),\text{d}\mu\right)$.
In view of (\ref{eq:h1_sum_form-1}), we have the formula 
\[
h^{(f)}(x)=\sum_{n=0}^{\infty}H_{n,0}^{(f)}\hat{F}_{n}(x)=\sum_{n=0}^{\infty}\frac{f_{n}(-x)}{4^{n}n!(n+1)!}=\frac{4}{\pi}\int_{0}^{K}\cos(\sqrt{x}u)\cn(u)\,\text{d}u.
\]
This formula in combination with the Fourier series (\ref{eq:four_cn})
allows us to evaluate $h^{(f)}(x)$ at the spectral points $\lambda_{m}$
of the Jacobi matrix $J^{(f)}$, as given in (\ref{eq:mu_carlitz1}).
The obtained values read
\[
h^{(f)}(\lambda_{m})=\frac{4}{k}\,\frac{q^{m+1/2}}{1+q^{2m+1}}\,,\ \,m\geq0,
\]
and these are in fact the eigenvalues of $H^{(f)}$. Respective eigenvectors
and their norms can be derived exactly in the same way as described
in the beginning of the proof of Theorem~\ref{thm:hank3456}.

(ii)~~Now we consider the entries $\alpha_{n}$, $\beta_{n}$ in
(\ref{eq:alpha-beta-gen}) for the values of parameters
\[
a=1,\;b=1/2,\;c=1/2,\;\sigma=2,\ \mbox{whence}\ \,\xi=1,\;\eta=3,\ \omega(\xi,\eta,\sigma)=1.
\]
Then the weight given in (\ref{eq:weight_gener}) equals
\[
w_{n}^{(g)}=\frac{\left(\frac{3}{2}\right)_{\!n}}{\sqrt{n!\,(2)_{n}}}=\frac{\sqrt{n+1}}{2^{2n}}\,\binom{2n+1}{n},
\]
and the square summable solution (\ref{eq:hn-plus}) of equation (\ref{eq:three-term-eq}),
as described in Proposition~\ref{thm:hn-plus}, equals
\[
h_{n}^{(g)}=\frac{k^{n}\Gamma(n+2)}{\Gamma\big(n+3\big)}\,\HG{n+2,0}{n+3}{k^{2}}=\frac{k^{n}}{n+2}\,.
\]
Thus $H_{m,n}^{(g)}=w_{m}^{(g)}w_{n}^{(g)}h_{m+n}^{(g)}$ coincides
with (\ref{eq:def_hank2}). By Lemma~\ref{lem:diff_eq}, the weighted
Hankel matrix $H^{(g)}$ commutes with $J=J^{(g)}$ introduced in
(\ref{eq:jacobi_mat}) where we let
\[
\alpha_{n}^{(g)}:=-2k(2n+3)\sqrt{(n+1)(n+2)}\,,\ \beta_{n}^{(g)}:=(k^{2}+1)(2n+2)^{2}.
\]
From (\ref{eq:recur_carlitz2}) it is seen that the Jacobi matrix
$J^{(g)}$ corresponds to the Family~\#2 of the Stieltjes-Carlitz
polynomials.

Similarly as in the forgoing case one can argue that $\sum_{m,n=0}^{\infty}H_{m,n}^{(g)}<\infty$
and therefore $H^{(g)}$ is a trace class operator.

In view of (\ref{eq:og_rel_carlitz2}), the functions
\[
\hat{G}_{n}(x):=\frac{g_{n}(-x)}{k^{n}(2n+1)!\sqrt{n+1}}\,,\ \,n\in\mathbb{N}_{0},
\]
form an orthonormal basis in $L^{2}\left((0,\infty),\text{d}\mu\right)$
where $\mu$ is defined in (\ref{eq:mu_carlitz2}). Using $\{\hat{G}_{n}\}$
we can again construct a common eigenbasis for both $J^{(g)}$ and
$H^{(g)}$ and, consequently, a unitary mapping $U$ such that $UJ^{(g)}U^{-1}$
is a multiplication operator by $x$ and $UH^{(g)}U^{-1}$ is a multiplication
operator by a function $h^{(g)}(x)$, both acting on $L^{2}\left((0,\infty),\text{d}\mu\right)$.
According to (\ref{eq:h2_sum_form-1}), we have
\begin{eqnarray*}
h^{(g)}(x) & = & \sum_{n=0}^{\infty}H_{n,0}^{(g)}\hat{G}_{n}(x)\,=\,\sum_{n=0}^{\infty}\frac{g_{n}(-x)}{4^{n}n!(n+2)!}\\
 & = & -\frac{8\sqrt{1-k^{2}}}{\pi k^{2}}\frac{\sin(\sqrt{x}K)}{\sqrt{x}}+\frac{8}{\pi k^{2}}\int_{0}^{K}\cos(\sqrt{x}u)\dn(u)\,\text{d}u.
\end{eqnarray*}
With the aid of this formula and (\ref{eq:four_dn}) we can evaluate
$h^{(g)}(x)$ at the spectral points $\lambda_{m}$ of the Jacobi
matrix $J^{(g)}$, as given in (\ref{eq:mu_carlitz2}). The obtained
values read
\[
h^{(g)}(\lambda_{m})=\frac{8}{k^{2}}\,\frac{q^{m}}{1+q^{2m}}\,,\ \,m\geq1,
\]
and these are in fact the eigenvalues of $H^{(g)}$. Note that in
this case, too, $\lambda_{0}=0$ is not an eigenvalue of $J^{(g)}$
and $h^{(g)}(0)$ is not an eigenvalue of $H^{(g)}$.

Respective eigenvectors and their norms can be derived as in the forgoing
cases. \end{proof}

\section{Some more weighted Hankel matrices\label{sec:Some-more-weighted}}

Another set of explicitly diagonalizable weighted Hankel matrices
can be obtained by permuting the parameters $a$, $b$, and $c$.
In fact, note that permuting $a$, $b$, and $c$ does not change
the Jacobi matrix defined in (\ref{eq:jacobi_mat}), (\ref{eq:alpha-beta-gen}),
but this need not be the case for the weight $w_{n}$ defined in (\ref{eq:weight_gener}).
Hence while keeping $J$ fixed we can get a new weighted Hankel matrix
lying in the commutant of $J$.

First we apply this observation to the Jacobi matrices $J^{(q)}$
and $J^{(s)}$ corresponding to Families \#4 and \#6 of the Stieltjes-Carlitz
polynomials, respectively. In both cases we can put $a=b=1/2$, $c=0$
(compare (\ref{eq:alpha-beta-gen}) to (\ref{eq:alpha-beta-symm})).
Then the weight in (\ref{eq:weight_gener}) is trivial, $w_{n}=1$.
Another possible choice of the parameters, $a=0$, $b=c=1/2$, leads
to a nontrivial weight but with the Jacobi matrix remaining untouched.
The result is described in detail in Theorem~\ref{thm:hank4a6a}
below.

Naturally, one can attempt to apply the same procedure to the Jacobi
matrices $J^{(p)}$ and $J^{(r)}$ corresponding to the Families \#3
and \#5 of the Stieltjes-Carlitz polynomials, respectively. Now we
put, in both cases, $a=b=-1/2$, $c=0$, and consequently we have
$w_{n}=1$. Unfortunately, the permutation $a=0$, $b=c=-1/2$, does
not yield  results of notable interest. As far as $J^{(p)}$ is concerned,
we were not able, for the moment, to evaluate eigenvalues of the newly
obtained weighted Hankel matrix explicitly. As for $J^{(r)}$, the
new weighted Hankel matrix turns out to be of rank $1$ and therefore
of little interest.

Let us introduce another couple of weighted Hankel matrices,
\begin{eqnarray*}
H_{m,n}^{(q')} & := & (2m+1)(2n+1)\binom{2m}{m}\binom{2n}{n}\!\left(\frac{k}{4}\right)^{\!m+n}\frac{1+(1-k^{2})(m+n+1)}{(m+n+1)(m+n+2)}\,,\\
\noalign{\smallskip}H_{m,n}^{(s')} & := & \binom{2m}{m}\binom{2n}{n}\!\left(\frac{k}{4}\right)^{m+n}\!\frac{(2m+1)(2n+1)}{(m+n+1)(m+n+2)}\,,
\end{eqnarray*}
$m,n\in\mathbb{N}_{0}$.

\begin{theorem}\label{thm:hank4a6a} The weighted Hankel matrices
$H^{(q')}$ and $H^{(s')}$ represent both positive trace class operators
on $\ell^{2}(\mathbb{N}_{0})$ with simple eigenvalues. We have:

\noindent (i)~~Eigenvalues of $H^{(q')}$ enumerated in descending
order are
\[
{\displaystyle \nu_{m}^{(q')}=\frac{2\pi}{kK}\,\frac{(2m+1)q^{m+1/2}}{1-q^{2m+1}}}\,,\ m\geq0.
\]
\noindent For an eigenvector corresponding to $\nu_{m}^{(q')}$ one
can choose the vector $\Psi_{m}^{(q)}$, as introduced in Theorem~\ref{thm:hank3456}.

\noindent (ii)~~Eigenvalues of $H^{(s')}$ enumerated in descending
order are
\[
{\displaystyle \nu_{m}^{(s')}=\frac{4\pi}{k^{2}K}\,\frac{mq^{m}}{1-q^{2m}}}\,,\ m\geq1.
\]
\noindent For an eigenvector corresponding to $\nu_{m}^{(s')}$ one
can choose the vector $\Psi_{m}^{(s)}$, as introduced in Theorem~\ref{thm:hank3456}.
\end{theorem}

\begin{proof} The proof again follows the same scheme as in the proof
of Theorem~\ref{thm:hank3456} or Theorem~\ref{thm:hank12}. Most
steps are quite routine and thus we confine ourselves to pointing
out only some features which are particular for the matrices in question.

In both cases we have $\sum_{m,n=0}^{\infty}H_{m,n}<\infty$ implying
that the matrices represent trace class operators.

As already mentioned, in both cases we let $a=0$, $b=c=1/2$. Hence,
comparing (\ref{eq:diff_h}) to (\ref{eq:three-term-eq}), we have
to put $\xi=a=0$, $\eta=b+c+2=3$. Also the weight (\ref{eq:weight_gener})
is the same in both cases,
\[
w_{n}=\frac{1}{n!}\left(\frac{3}{2}\right)_{\!n}=\frac{2n+1}{2^{2n}}\,\binom{2n}{n}.
\]

(i)~~The choice of $\sigma$ corresponds to the Jacobi matrix $J^{(q)}$,
see the proof of Theorem~\ref{thm:hank3456} ad (ii). We have
\[
\sigma=\frac{1+2k^{2}}{1+k^{2}},\ \mbox{whence}\ \,\omega(\xi,\eta,\sigma)=1,
\]
see (\ref{eq:b-x_y_s}). The square summable solution (\ref{eq:hn-plus})
of equation (\ref{eq:three-term-eq}) equals
\[
h_{n}^{(q')}=\frac{k^{n}}{n+1}\,\HG{n+1,-1}{n+2}{k^{2}}=k^{n}\,\frac{n+2-(n+1)k^{2}}{(n+1)(n+2)}\,.
\]

$H^{(q')}$ can be diagonalized by the same unitary transform as the
matrix $H^{(q)}$ which is described in the proof of Theorem~\ref{thm:hank3456}
ad (ii). Using some routine manipulations, quite similarly as in the
proof of Proposition~\ref{thm:sum-SCpoly-3456}, one can transform
$H^{(q')}$ to a multiplication operator by the function
\begin{eqnarray*}
h^{(q')}(x) & = & \sum_{n=0}^{\infty}H_{n,0}^{(q')}\hat{Q}_{n}(x)\,=\,\sum_{n=0}^{\infty}\frac{(3/2)_{n}}{n!}\left[\frac{1-k^{2}}{n+1}+\frac{k^{2}}{(n+1)(n+2)}\right]\frac{(-1)^{n}q_{n}(x)}{(2n+1)!}\\
 & = & \frac{4}{\pi}\int_{0}^{K}\frac{\sin(\sqrt{x}u)}{\sqrt{x}}\left[(1+k^{2})\sn(u)-2k^{2}\sn^{3}(u)\right]\text{d}u.
\end{eqnarray*}

This formula in combination with the Fourier series (\ref{eq:four_sn})
and (\ref{eq:four_sn_cubed}) makes it possible to evaluate $h^{(q')}(x)$
at the spectral points $\lambda_{m}$ of the Jacobi matrix $J^{(q)}$,
as given in (\ref{eq:mu_carlitz34}). The obtained values read
\[
h^{(q')}(\lambda_{m})=\frac{2\pi}{kK}\frac{(2m+1)q^{m+1/2}}{1-q^{2m+1}}\,,\ \,m\geq0,
\]
and these are in fact the eigenvalues of $H^{(q')}$.

(ii)~~The choice of $\sigma$ corresponds to the Jacobi matrix $J^{(s)}$,
see the proof of Theorem~\ref{thm:hank3456} ad (iv). We have
\[
\sigma=\frac{2+k^{2}}{1+k^{2}},\ \mbox{whence}\ \,\omega(\xi,\eta,\sigma)=2,
\]
see (\ref{eq:b-x_y_s}). The square summable solution (\ref{eq:hn-plus})
of equation (\ref{eq:three-term-eq}) equals
\[
h_{n}^{(s')}=\frac{k^{n}}{(n+1)(n+2)}\,\HG{n+1,0}{n+3}{k^{2}}=\frac{k^{n}}{(n+1)(n+2)}\,.
\]

$H^{(s')}$ can be diagonalized by the same unitary transform as the
matrix $H^{(s)}$ which is described in the proof of Theorem~\ref{thm:hank3456}
ad (iv). Using some routine manipulations, similarly as in the proof
of Proposition~\ref{thm:sum-SCpoly-3456}, one can transform $H^{(s')}$
to a multiplication operator by the function
\begin{eqnarray*}
h^{(s')}(x) & = & \sum_{n=0}^{\infty}H_{n,0}^{(s')}\hat{S}_{n}(x)\,=\,\sum_{n=0}^{\infty}\frac{(3/2)_{n}}{(n+2)!}\frac{(-1)^{n}s_{n}(x)}{(2n+1)!}\\
 & = & \frac{4}{\pi}\!\left(\frac{\sin(\sqrt{x}K)}{\sqrt{x}}-\int_{0}^{K}\cos(\sqrt{x}u)\sn^{2}(u)\,\text{d}u\right)\!.
\end{eqnarray*}

This formula in combination with the Fourier series (\ref{eq:four_sn_squared})
makes it possible to evaluate $h^{(s')}(x)$ at the spectral points
$\lambda_{m}$ of the Jacobi matrix $J^{(s)}$, as given in (\ref{eq:mu_carlitz56}).
The obtained values read
\[
h^{(s')}(\lambda_{m})=\frac{4\pi}{k^{2}K}\,\frac{mq^{m}}{1-q^{2m}}\,,\ \,m\geq1,
\]
and these are in fact the eigenvalues of $H^{(s')}$. \end{proof}

Finally we consider permutations of the parameters $a$, $b$ and
$c$ in case of the matrices $H^{(f)}$ and $H^{(g)}$ from Theorem~\ref{thm:hank12}.
As for $H^{(f)}$, the original values were $a=0$, $b=1/2$, $c=-1/2$.
Permuting $a$, $b$, $c$ in (\ref{eq:weight_gener}) provides us
with two new weights. In case of $H^{(g)}$ we have $a=1$, $b=c=1/2$,
and permuting $a$, $b$, $c$ leads to just one new weight.

These considerations lead us to introducing three weighted Hankel
matrices
\begin{eqnarray*}
H_{m,n}^{(f')} & := & \frac{k^{m+n}\Gamma(m+n+3/2)}{\sqrt{(2m+1)(2n+1)}(m+n+1)!}\,\HG{m+n+3/2,1/2}{m+n+2}{k^{2}}\\
 & = & \frac{2k^{m+n}}{\sqrt{\pi(2m+1)(2n+1)}}\int_{0}^{1}\frac{x^{2m+2n+2}\,\text{d}x}{\sqrt{(1-x^{2})(1-k^{2}x^{2})}}\,,
\end{eqnarray*}
\begin{eqnarray*}
H_{m,n}^{(f'')} & := & \frac{k^{m+n}\sqrt{(2m+1)(2n+1)}\Gamma(m+n+1/2)}{(m+n+1)!}\,\HG{m+n+1/2,-1/2}{m+n+2}{k^{2}}\\
\noalign{\smallskip} & = & \frac{4k^{m+n}\sqrt{(2m+1)(2n+1)}}{\sqrt{\pi}}\int_{0}^{1}x^{2m+2n}\sqrt{(1-x^{2})(1-k^{2}x^{2})}\,\text{d}x,
\end{eqnarray*}
\begin{eqnarray*}
H_{m,n}^{(g')} & := & \frac{k^{m+n}\sqrt{(m+1)(n+1)}\Gamma(m+n+3/2)}{(m+n+2)!}\,\HG{m+n+3/2,-1/2}{m+n+3}{k^{2}}\\
\noalign{\smallskip} & = & \frac{4k^{m+n}\sqrt{(m+1)(n+1)}}{\sqrt{\pi}}\int_{0}^{1}x^{2m+2n+2}\sqrt{(1-x^{2})(1-k^{2}x^{2})}\,\text{d}x,
\end{eqnarray*}
$m,n\in\mathbb{N}_{0}$. Here we have again used the integral representation
(\ref{eq:gauss_hyp_int_repre}).

\begin{theorem}\label{thm:hank1ab2a} Each of the weighted Hankel
matrices $H^{(f')}$, $H^{(f'')}$ and $H^{(g')}$ represents a positive
trace class operator on $\ell^{2}(\mathbb{N}_{0})$ with simple eigenvalues.
We have:

\noindent (i)~~Eigenvalues of $H^{(f')}$ enumerated in descending
order are
\[
{\displaystyle \nu_{m}^{(f')}=\frac{4K}{\sqrt{\pi}k}\,\frac{q^{m+1/2}}{(2m+1)(1-q^{2m+1})}}\,,\ m\geq0.
\]
\noindent For an eigenvector corresponding to $\nu_{m}^{(f')}$ one
can choose the vector $\Psi_{m}^{(f)}$, as introduced in Theorem~\ref{thm:hank12}
ad (i).

\noindent (ii)~~Eigenvalues of $H^{(f'')}$ enumerated in descending
order are
\[
{\displaystyle \nu_{m}^{(f'')}=\frac{2\pi^{3/2}}{kK}\,\frac{(2m+1)q^{m+1/2}}{1-q^{2m+1}}}\,,\ m\geq0.
\]
\noindent For an eigenvector corresponding to $\nu_{m}^{(f'')}$
one can choose the vector $\Psi_{m}^{(f)}$, as introduced in Theorem~\ref{thm:hank12}
ad (i).

\noindent (iii)~~Eigenvalues of $H^{(g')}$ enumerated in descending
order are
\[
{\displaystyle \nu_{m}^{(g')}=\frac{2\pi^{3/2}}{k^{2}K}\,\frac{mq^{m}}{1-q^{2m}}}\,,\ m\geq1.
\]
\noindent For an eigenvector corresponding to $\nu_{m}^{(g')}$ one
can choose the vector $\Psi_{m}^{(g)}$, as introduced in Theorem~\ref{thm:hank12}
ad (ii). \end{theorem}

\begin{proof} Similarly as in the proof of Theorem~\ref{thm:hank4a6a}
we confine ourselves to pointing out only some features which are
particular for the matrices treated in this theorem. Otherwise the
basic scheme is still the same as in the proof of Theorem~\ref{thm:hank3456}
or Theorem~\ref{thm:hank12}.

It is straightforward to verify in each of the three cases that $\sum_{m,n=0}^{\infty}H_{m,n}<\infty$
which implies that the matrices represent trace class operators.

(i)~~In this case $a=1/2$, $b=0$ and $c=-1/2$. The choice of
$\sigma=1$ corresponds to the Jacobi matrix $J^{(f)}$, see the proof
of Theorem~\ref{thm:hank12} ad (i). We put $\xi=a=1/2$, $\eta=b+c+2=3/2$,
and then $\omega(\xi,\eta,\sigma)=1/2$, see (\ref{eq:b-x_y_s}).
The weight (\ref{eq:weight_gener}) equals
\[
w_{n}^{(f')}=\sqrt{\frac{(1)_{n}\left(\frac{1}{2}\right)_{\!n}}{n!\left(\frac{3}{2}\right)_{\!n}}}=\frac{1}{\sqrt{2n+1}}\,.
\]
Furthermore, the square summable solution (\ref{eq:hn-plus}) of equation
(\ref{eq:three-term-eq}) equals
\[
h_{n}^{(f')}=\frac{k^{n}\Gamma(n+3/2)}{(n+1)!}\,\HG{n+3/2,1/2}{n+2}{k^{2}}=\frac{2k^{n}}{\sqrt{\pi}}\int_{0}^{1}\frac{x^{2n+2}\,\text{d}x}{\sqrt{(1-x^{2})(1-k^{2}x^{2})}}\,.
\]

$H^{(f')}$ can be diagonalized by the same unitary transform as the
matrix $H^{(f)}$ which is described in the proof of Theorem~\ref{thm:hank12}
ad (i). Using some routine manipulations, quite similarly as in the
proof of Proposition~\ref{thm:sum-SCpoly-3456}, one can transform
$H^{(f')}$ to a multiplication operator by the function
\begin{eqnarray*}
h^{(f')}(x) & = & \sum_{n=0}^{\infty}H_{n,0}^{(f')}\hat{F}_{n}(x)\,=\,\frac{2}{\sqrt{\pi}}\sum_{n=0}^{\infty}\frac{f_{n}(-x)}{(2n+1)!}\int_{0}^{K}\sn^{2n+2}(u)\,\text{d}u\\
 & = & \frac{2}{\sqrt{\pi}}\int_{0}^{K}\frac{\sin(\sqrt{x}u)}{\sqrt{x}}\sn(u)\,\text{d}u.
\end{eqnarray*}

This formula in combination with the Fourier series (\ref{eq:four_sn})
makes it possible to evaluate $h^{(f')}(x)$ at the spectral points
$\lambda_{m}$ of the Jacobi matrix $J^{(f)}$, as given in (\ref{eq:mu_carlitz1}).
The obtained values read
\[
h^{(f')}(\lambda_{m})=\frac{4K}{\sqrt{\pi}k}\,\frac{q^{m+1/2}}{(2m+1)(1-q^{2m+1})}\,,\ \,m\geq0,
\]
and these are in fact the eigenvalues of $H^{(f')}$.

(ii)~~In this case $a=-1/2$, $b=0$ and $c=1/2$. The choice of
$\sigma=1$ again corresponds to the Jacobi matrix $J^{(f)}$, see
the proof of Theorem~\ref{thm:hank12} ad (i). We put $\xi=a=-1/2$,
$\eta=b+c+2=5/2$, and then $\omega(\xi,\eta,\sigma)=3/2$, see (\ref{eq:b-x_y_s}).
The weight (\ref{eq:weight_gener}) equals
\[
w_{n}^{(f'')}=\sqrt{\frac{(1)_{n}\left(\frac{3}{2}\right)_{\!n}}{n!\left(\frac{1}{2}\right)_{\!n}}}=\sqrt{2n+1}.
\]
Furthermore, the square summable solution (\ref{eq:hn-plus}) of equation
(\ref{eq:three-term-eq}) equals
\begin{eqnarray*}
h_{n}^{(f'')} & = & \frac{k^{n}\Gamma(n+1/2)}{(n+1)!}\,\HG{n+1/2,-1/2}{n+2}{k^{2}}\\
 & = & \frac{4k^{n}}{\sqrt{\pi}}\int_{0}^{1}x^{2n}\sqrt{(1-x^{2})(1-k^{2}x^{2})}\,\text{d}x.
\end{eqnarray*}

$H^{(f'')}$ can be diagonalized by the same unitary transform as
the matrix $H^{(f)}$ which has been described in the proof of Theorem~\ref{thm:hank12}
ad (i). After some routine manipulations, similarly as in the proof
of Proposition~\ref{thm:sum-SCpoly-3456}, one can transform $H^{(f'')}$
to a multiplication operator by the function
\begin{eqnarray*}
h^{(f'')}(x) & = & \sum_{n=0}^{\infty}H_{n,0}^{(f'')}\hat{F}_{n}(x)\,=\,\frac{4}{\sqrt{\pi}}\sum_{n=0}^{\infty}\frac{f_{n}(-x)}{(2n)!}\int_{0}^{K}\sn^{2n}(u)\cn^{2}(u)\dn^{2}(u)\,\text{d}u\\
 & = & \frac{4}{\sqrt{\pi}}\left(\cos(\sqrt{x}K)+\sqrt{x}\int_{0}^{K}\sin(\sqrt{x}u)\sn(u)\,\text{d}u\right)\!.
\end{eqnarray*}

This formula in combination with the Fourier series (\ref{eq:four_sn})
makes it possible to evaluate $h^{(f'')}(x)$ at the spectral points
$\lambda_{m}$ of the Jacobi matrix $J^{(f)}$, as given in (\ref{eq:mu_carlitz1}).
The obtained values read
\[
h^{(f'')}(\lambda_{m})=\frac{2\pi^{3/2}}{kK}\,\frac{(2m+1)q^{m+1/2}}{1-q^{2m+1}}\,,\ \,m\geq0,
\]
and these are in fact the eigenvalues of $H^{(f'')}$.

(iii)~~In this case $a=b=1/2$, $c=1$. The choice of $\sigma=2$
corresponds to the Jacobi matrix $J^{(g)}$, see the proof of Theorem~\ref{thm:hank12}
ad (ii). We put $\xi=a=1/2$, $\eta=b+c+2=7/2$, and then $\omega(\xi,\eta,\sigma)=3/2$,
see (\ref{eq:b-x_y_s}). The weight (\ref{eq:weight_gener}) equals
$w_{n}^{(g')}=\sqrt{n+1}$. Furthermore, the square summable solution
(\ref{eq:hn-plus}) of equation (\ref{eq:three-term-eq}) equals
\begin{eqnarray*}
h_{n}^{(g')} & = & \frac{k^{n}\Gamma(n+3/2)}{(n+2)!}\,\HG{n+3/2,-1/2}{n+3}{k^{2}}\\
 & = & \frac{4k^{n}}{\sqrt{\pi}}\int_{0}^{1}x^{2n+2}\sqrt{(1-x^{2})(1-k^{2}x^{2})}\,\text{d}x.
\end{eqnarray*}

$H^{(g')}$ can be diagonalized by the same unitary transform as the
matrix $H^{(g)}$ which has been described in the proof of Theorem~\ref{thm:hank12}
ad (ii). After some routine manipulations, similarly as in the proof
of Proposition~\ref{thm:sum-SCpoly-3456}, one can transform $H^{(g')}$
to a multiplication operator by the function
\begin{eqnarray*}
h^{(g')}(x) & = & \sum_{n=0}^{\infty}H_{n,0}^{(g')}\hat{G}_{n}(x)\,=\,\frac{4}{\sqrt{\pi}}\sum_{n=0}^{\infty}\frac{g_{n}(-x)}{(2n+1)!}\int_{0}^{K}\sn^{2n+2}(u)\cn^{2}(u)\dn^{2}(u)\,\text{d}u\\
 & = & \frac{1}{\sqrt{\pi}}\left(\frac{\sin(\sqrt{x}K)}{\sqrt{x}}-\int_{0}^{K}\cos(\sqrt{x}u)\sn^{2}(u)\,\text{d}u\right)\!.
\end{eqnarray*}

This formula in combination with the Fourier series (\ref{eq:four_sn_squared})
makes it possible to evaluate $h^{(g')}(x)$ at the spectral points
$\lambda_{m}$ of the Jacobi matrix $J^{(g)}$, as given in (\ref{eq:mu_carlitz2}).
The obtained values read
\[
h^{(g')}(\lambda_{m})=\frac{2\pi^{3/2}}{k^{2}K}\,\frac{mq^{m}}{1-q^{2m}}\,,\ \,m\geq1,
\]
and these are in fact the eigenvalues of $H^{(g')}$. \end{proof}

\section*{Acknowledgments}

The authors acknowledge financial support by the Ministry of Education, Youth and Sports of the Czech Republic project no.~CZ.02.1.01/0.0/0.0/16\_019/0000778.

\end{document}